\documentclass[reqno]{amsart}
\usepackage{diagrams}
\usepackage{amssymb}
\usepackage{mathrsfs}
\usepackage{cite}
\usepackage{tikz-cd}
\usepackage{MnSymbol}
\usepackage{a4wide}
 
\usepackage{xcolor}

\DeclareMathOperator\C{\mathbb C}
\DeclareMathOperator\R{\mathbb R}
\DeclareMathOperator\Z{\mathbb Z}
\newcommand{\Om}{\Omega}
\newcommand{\ber}{\operatorname{ber}}
\newcommand{\G}{{\mathbb G}}

\newtheorem{theorem}{Theorem}[section]
\newtheorem{lemma}[theorem]{Lemma}
\newtheorem{cor}[theorem]{Corollary}
\newtheorem{prop}[theorem]{Proposition}
\theoremstyle{definition}
\newtheorem{definition}[theorem]{Definition}

\theoremstyle{remark}
\newtheorem{remark}[theorem]{Remark}

\newcommand{\dontprint}[1]\relax

\newcommand{\La}{\Lambda}
\newcommand{\Ga}{\Gamma}
\renewcommand{\P}{{\mathbb P}}
\newcommand{\A}{{\mathbb A}}
\newcommand{\wt}{\widetilde}
\newcommand{\ot}{\otimes}
\newcommand{\bos}{\operatorname{bos}}
\newcommand{\can}{\operatorname{can}}
\newcommand{\Hom}{\operatorname{Hom}}
\newcommand{\Ext}{\operatorname{Ext}}
\renewcommand{\AA}{{\mathcal A}}

\newcommand{\DD}{{\mathcal D}}
\newcommand{\CC}{{\mathcal C}}

\newcommand{\FF}{{\mathcal F}}
\newcommand{\GG}{{\mathcal G}}
\newcommand{\II}{{\mathcal I}}
\newcommand{\LL}{{\mathcal L}}
\newcommand{\MM}{{\mathcal M}}
\newcommand{\NN}{{\mathcal N}}

\newcommand{\OO}{{\mathcal O}}
\newcommand{\UU}{{\mathcal U}}
\newcommand{\PP}{{\mathcal P}}
\renewcommand{\SS}{{\mathcal S}}
\newcommand{\QQ}{{\mathcal Q}}
\newcommand{\VV}{{\mathcal V}}
\newcommand{\WW}{{\mathcal W}}
\newcommand{\XX}{{\mathcal X}}
\newcommand{\YY}{{\mathcal Y}}

\newcommand{\de}{\delta}
\newcommand{\sub}{\subset}
\newcommand{\Pfaff}{\operatorname{Pf}}
\newcommand{\Spec}{\operatorname{Spec}}
\newcommand{\Res}{\operatorname{Res}}
\newcommand{\Ann}{\operatorname{Ann}}
\newcommand{\Det}{\operatorname{Det}}
\newcommand{\Ber}{\operatorname{Ber}}
\newcommand{\Cone}{\operatorname{Cone}}
\newcommand{\lan}{\langle}
\newcommand{\ran}{\rangle}
\newcommand{\ov}{\overline}
\newcommand{\im}{\operatorname{im}}
\newcommand{\om}{\omega}
\newcommand{\la}{\lambda}
\renewcommand{\a}{\alpha}
\renewcommand{\b}{\beta}
\newcommand{\ga}{\gamma}

\newcommand{\id}{\operatorname{id}}
\newcommand{\und}{\underline}
\renewcommand{\th}{\theta}
\newcommand{\coker}{\operatorname{coker}}
\newcommand{\red}{{\operatorname{red}}}
\newcommand{\hra}{\hookrightarrow}
\newcommand{\we}{\wedge}
\newcommand{\De}{\Delta}
\newcommand{\rk}{\operatorname{rk}}

\newcommand{\fm}{{\frak m}}
\newcommand{\Gr}{\operatorname{Gr}}
\newcommand{\per}{\operatorname{per}}
\newcommand{\eps}{\epsilon}

\numberwithin{equation}{section}

\title[Supermeasure and superperiod map]
{Regularity of the superstring supermeasure and the superperiod map}
\author{Giovanni Felder}
\address{Department of mathematics,
ETH Zurich, 8092 Zurich, Switzerland}
\email{felder@math.ethz.ch}
\author{David Kazhdan}
\address{Einstein Institute of Mathematics,
The Hebrew University of Jerusalem,
Jerusalem 91904, Israel}
\email{kazhdan@math.huji.ac.il}
\author{Alexander Polishchuk}
\address{
    Department of Mathematics, 
    University of Oregon, 
    Eugene, OR 97403, USA; National Research University Higher School of Economics; and Korea Institute for 
    Advanced Study 
  }
  \email{apolish@uoregon.edu}
\begin{document}
\maketitle
\begin{abstract}
  The supermeasure whose integral is the genus $g$ vacuum amplitude of
  superstring theory is potentially singular on the locus in the
  moduli space of supercurves where the corresponding even
  theta-characteristic has nontrivial sections. We show that the
  supermeasure is actually regular for $g\leq 11$. The result relies
  on the study of the superperiod map. We also show that the minimal
  power of the classical Schottky ideal that annihilates the image of
  the superperiod map is equal to $g$ if $g$ is odd and is equal to $g$
  or $g-1$ if $g$ is even.
\end{abstract}
\section{Introduction}
The moduli space of supercurves $\SS=\SS_g$ (of given genus $g$) has been
studied in many papers, including \cite{CraneRabin1988, Deligne1987,
  DRS, LeBrunRothstein1988}, motivated by superstring perturbation
theory, see \cite{Witten} for a recent account with references to the
original literature.  The space $\SS$ is a
superstack of complex dimension $3g-3|2g-2$. The corresponding reduced
stack $\SS_{\bos}$ (where ``bos" stands for ``bosonic") is the moduli stack of pairs $(C,L)$, where $C$ is a
usual smooth projective curve of genus $g$ and $L$ is a
theta-characteristic on $C$ (also known as spin structure).  The space
$\SS_{\bos}$ (and hence $\SS$) has two components, which we call {\it even}
and {\it odd}, that are distinguished by the condition that $L$ is
even or odd.

In the first part of this paper we study the behavior of the
superstring supermeasure on the even component $\SS^+$ of the moduli
space of supercurves $\SS$ (the supermeasure vanishes on the odd
component).  Superstring perturbation theory is well-understood up to
genus 2, see \cite{DHokerPhong2002-1,DHokerPhong2002-2}. In this case the
supermoduli spaces are canonically projected, so that it is possible
to first integrate the supermeasure over odd variables to get a
measure on the moduli space of spin curves and a description in
terms of classical geometry. In higher genus there is no
projection to the reduced space \cite{DonagiWitten2012}
but there is also an additional difficulty:
the
supermeasure is a priori only defined on a Zariski open subset of the
moduli space where the corresponding theta-characteristic has no
nontrivial section. E. Witten recently revisited and extended the
study of the supermeasure in \cite{Witten-hol-str} and argued that the
supermeasure actually extends smoothly to the whole moduli space for
low genus, indicating potential difficulties starting at
genus 11.

Our main result is the regularity of this supermeasure
for $g\le 11$ along the locus where the corresponding
theta-characteristic acquires a section (see Theorem \ref{s5-thm}),
confirming Witten's expectation.

More precisely, the contribution of genus $g$ to the vacuum amplitude
of type II superstring theory is given as the integral of the supermeasure
(section of the Berezinian of the cotangent bundle) over a suitable cycle
in $\SS^+\times\ov{\SS^+}$
(where the bar denotes the complex conjugate space).
This supermeasure can be expressed in
holomorphic terms in terms of the super analogue of the Mumford
isomorphism, constructed in \cite{Voronov}, see
\cite{RoslySchwarzVoronov1989}.  The relevant special case of the
Mumford isomorphism is a canonical isomorphism
$$\operatorname{Ber}^5(R\pi_*\omega_{X/\SS})
\cong\operatorname{Ber}(T^*_\SS)$$ 
of the fifth tensor power of the
Berezinian of the derived direct image
of the relative dualizing sheaf (Berezinian)
$\omega_{X/\SS}$ of the universal supercurve $\pi\colon X\to\SS$ with the
Berezinian of the holomorphic cotangent bundle.

The supermeasure is defined on a certain open subset of the product
$\SS^+\times \ov{\SS^+}$.  In the corresponding reduced space
$\SS_{\bos}^+\times \ov{\SS_{\bos}^+}$ we have a natural locus of pairs of
curves with theta-characteristics $((C,L),(C',L'))$ such that $C=C'$,
which we call the {\it quasi-diagonal}.  The supermeasure can be
defined as a meromorphic section of the holomorphic Berezinian on a
neighborhood of the quasi-diagonal in $\SS^+\times\ov{\SS^+}$.  It is
the image of the canonical section of
$$\operatorname{Ber}^5(R\pi_*\omega_{X/\SS})\boxtimes
\operatorname{Ber}^5(R\pi_*\omega_{\ov{X}/\ov{\SS}})$$ 
induced by the
embedding of $\pi_*\omega_{X/\SS}$ into the local system
$R^1\pi_*\pi^{-1}\mathcal O_{\SS}$, which has a canonical pairing.
The possible problem with the regularity occurs when the higher direct
images $R^i\pi_*\omega_{X/\SS}$ are not locally free. This happens
near points $(C,L)$ such that $H^0(C,L)\neq 0$. We show in Theorem
\ref{s5-thm} that the supermeasure extends to a regular section for
$g\le 11$.

Note that if we view $\SS^+\times\ov{\SS^+}$ as a real supermanifold
then a section of the holomorphic Berezinian can be viewed as a codimension $6g-6|0$ closed integral form on it.
Thus, we have to extend the quasi-diagonal to a sub-supermanifold $\Delta $ in $\SS^+\times\ov{\SS^+}$ of real dimension $6g-6|4g-4$.
We can integrate over $\Delta $ the restriction of the holomorphic Berezinian provided the supermeasure is regular in its neighborhood. The integral does not depend on a choice of $\Delta $ since the integral form is closed 
(there are convergence issues at infinity, which we do not consider in this paper).


In the second part of the paper we study the superperiod map at the generic point of the
component $\SS^+$.
The superperiod map gives rise to a natural {\it super-Schottky ideal} $\II_{s-Sch}$ on the period domain,
which is containd in the usual Schottky ideal $\II_{Sch}$ describing the image of the classical period map.
The problem of studying $\II_{s-Sch}$ was raised in \cite[Rem.\ 6.10]{CV}.
Our second main result concerns the minimal power $d$ such that
$$\II_{Sch}^d\sub \II_{s-Sch}.$$ 
It is easy to see that $d\le g$. 
We prove that for $g$ odd, one has $d=g$, while for $g$ even, $d\ge g-1$.

The paper is organized as follows: in Section \ref{s-2} we review some
standard facts about supercurves and their moduli. In Section
\ref{loc-free-sec} we study the direct image $\pi_*\om_{X/S}$ and the higher direct image
$R^1\pi_*\om_{X/S}$ of the relative
Berezinian of the universal supercurve and show that they are locally
free on the open set where the even theta-characteristic has no
nontrivial sections and are not locally free away from it, as well as on the component of
moduli space with odd theta-characteristic. In the following Section
\ref{s-5} we study the behaviour of the period matrix as we approach
the divisor of even theta-characteristics with nontrivial
sections. The key idea is to realize the symplectic local system
$R^1\pi_*\mathbb C_{X/S}$, for a family of supercurves
$\pi\colon X\to S$, where $\mathbb C_{X/S}=\pi^{-1}\mathcal O_S$, as a symplectic reduction of a bigger symplectic
bundle $\VV$ with a Lagrangian subbundle $L_{\mathrm{can}}$ whose
reduction is $\pi_*\omega_{X/S}$. Away from the divisor of even
theta-characteristics with nontrivial sections, the Lagrangian
subbundle $\pi_*\omega_{X/S}$ of $R^1\pi_*\mathbb C_{X/S}$ is locally the graph of the
period matrix, a map $\Omega \colon\Lambda'\to \Lambda$ for some
Lagrangian splitting $R^1\pi_*\mathbb C_{X/S}=\Lambda\oplus
\Lambda'$. The singularities of $\Omega$ are governed by the map
$L_\Lambda\to \mathcal V/L_{\mathrm{can}}$ where the Lagrangian
subbundle $L_\Lambda\subset \mathcal V$ projects to $\Lambda$.  The
period matrix has a pole on the locus where this map fails to be an
isomorphism. We show that this locus is locally the zero set of a
Pfaffian $f$ and that $f\Omega$ is regular.  This result allows us to
compute the behavior of the supermeasure as we approach the locus,
leading to our regularity result.  In the remaining sections we study
the second variation of the period matrix in the odd directions at a
generic point of the moduli space of supercurves with even theta
characteristic. An analytic formula for the second variation was given
by D'Hoker and Phong in \cite{DHokerPhong1989} in terms of the Szeg\"o
kernel. We present in Section \ref{s-6} an alternative approach by
expressing the second variation as a Massey product.  We study this
Massey product in Section \ref{s-7} giving explicit formulas, in
particular for hyperelliptic curves. This formula is then used in
Section \ref{s-8} to prove our result on the super-Schottky ideal.

\subsection*{Conventions} 
We work with either complex algebraic or analytic superspaces. 
Whenever we work with the sheaf $\C_{X/S}$, we use classical topology and analytic framework.
For a superscheme $S$ we denote by $|S|$ the underlying topological space and by $S_{\bos}$ its bosonization, which is a usual scheme with
the underlying topological space $|S|$ and the structure sheaf $\OO_S/(\OO_S\cdot\OO_S^-)$ (the quotient by the ideal generated by odd functions).
By a {\it curve} we always mean a connected smooth projective curve over $\C$. When considering moduli spaces of curves we 
always assume that the genus $g$ of a curve is $\ge 2$.
By a {\it supercurve} we always mean a compact complex supermanifold  (or a proper smooth algebraic supervariety)
of dimension $1|1$ with a superconformal
structure (see Section \ref{ber-supercurve-sec} for details). These are often called {\it super Riemann surfaces}.
For a family of smooth curves $C/S$ (resp., supercurves $X/S$) we denote by $\om_{C/S}$ (resp., $\om_{X/S})$
the relative dualizing sheaf (given by the relative Berezinian in the supercase).
Since we work in the super-context throughout the paper, we drop 
the prefix ``super"  in {\it supervector bundles} and {\it supervector spaces}.

\subsection*{Acknowledgements}
G.F. is partially supported by the National Centre of
Competence in Research ``SwissMAP --- The Mathematics of Physics'' of the
Swiss National Science Foundation. He thanks the Hebrew University of Jerusalem, where part of this work was done, for hospitality.
D.K. is partially supported by the ERC under grant agreement 669655.
A.P. is partially supported by the NSF grant DMS-2001224, 
by the National Center of Competence in Research ``SwissMAP --- The Mathematics of Physics'' of the Swiss National Science Foundation, and 
by the Russian Academic Excellence Project `5-100' within the framework of the HSE University Basic Research Program.
While working on this project, A.P. was visiting ETH Zurich and Hebrew University of Jerusalem. 
He would like to thank these institutions for hospitality and excellent working conditions.

\section{Supercurves}\label{s-2}
In this section we review some standard material on supercurves (also
called super Riemann surfaces).

\subsection{Berezinian of a supercurve}\label{ber-supercurve-sec}

Let $\pi\colon X\to S$ be a smooth proper supercurve over a superscheme $S$. By
definition the tangent sheaf $T_{X/S}$ has a locally free
$\mathcal O_X$-submodule ${\mathcal D}$ of rank $(0,1)$ such that the composition
of the canonical projection with the Lie bracket
${\mathcal D}\otimes {\mathcal D}\to T_{X/S}/{\mathcal D}$, which is a map of $\mathcal O_X$-modules,
is an isomorphism. Thus, we have an exact sequence
\begin{equation}\label{e-00}
  0\to {\mathcal D}\to T_{X/S}\to {\mathcal D}^{\otimes 2}\to 0
\end{equation}
The dual sequence is
\begin{equation}\label{e-01}
  0\to ({\mathcal D}^\vee)^{\otimes 2}\to \Omega_{X/S}^1\to {\mathcal D}^\vee\to 0
\end{equation}
and taking Berezinians shows that
${\mathcal D}^\vee\cong \omega_{X/S}:=\operatorname{Ber}(\Omega_{X/S}^1)$.


Below we work with sheaves in the classical topology.  

The map $\Omega_{X/S}^1\to {\mathcal D}^\vee$ defines an
$\mathcal O_S$-linear derivation
\[
  \delta\colon \mathcal O_X\to \omega_{X/S}.
\]
The map $\de$ is surjective and its kernel is the sheaf $\mathbb C_{X/S}=\pi^{-1}\mathcal O_S$ of 
functions that are locally constant along the fibers.

Locally there exist fiber coordinates $z,\theta$ such that $\mathcal D$
is spanned by $D=\partial_\theta+\theta\partial_z$. In these coordinates
the map $\delta$ is $f\mapsto [dz|d\theta]D(f)$ where $[dz|d\theta]$ is
the local section of the Berezinian determined by the basis $dz,d\theta$ of $\Omega^1_{X/S}$.

The long exact sequence associated with
\begin{equation}\label{const-sh-resolution}
  0\to \mathbb C_{X/S}\to \mathcal O_X\rTo{\de} \omega_{X/S}\to 0
\end{equation}
includes 
\begin{equation}\label{e-les}
  \pi_*\omega_{X/S}\to R^1\pi_*\mathbb C_{X/S}\to R^1\pi_*\mathcal O_X\to R^1\pi_*\omega_{X/S}\to
  R^2\pi_*\mathbb C_{X/S}\to 0,
\end{equation}
(since $R^2\pi_*\mathcal O_X=0$).  Locally we have
$R^i\pi_*\mathbb C_{X/S}\cong H^i(|X_0|,\mathbb C)\otimes \mathcal O_S$,
where $|X_0|$ is the underlying topological space of a fibre
$X_0$. Thus, $R^1\pi_*\mathbb C_{X/S}\cong\mathcal O_S^{\oplus 2g}$ and
$R^2\pi_*\mathbb C_{X/S}\cong\mathcal O_S$. In fact, the latter isomorphism is global, and we have in particular a
surjective trace map
 \begin{equation}\label{Ber-curve-trace-map}
   \tau \colon R^1\pi_*\omega_{X/S}\to \mathcal O_S
 \end{equation}
that plays a role in duality theory (see Sec.\ \ref{hodge-duality-sec} below).

\subsection{Supercurves over an even base}\label{supercurves-even-base-sec}

Assume that $S$ is even. Then the decomposition of $\OO_X$ into the even and odd components gives a splitting of $X$:
$$\mathcal O_X\cong{\bigwedge}^\bullet_{\mathcal O_C}L=\mathcal O_C\oplus L$$
where $L$ is a relative theta-characteristic on the underlying usual family of curves $C\to
S$, with a fixed isomorphism
\begin{equation}\label{theta-char-str-map}
L\otimes L\to\omega_{C/S}.
\end{equation}
The dualizing sheaf $\omega_{X/S}$ is $\omega_C\oplus L$, with
the $\mathcal O_X$-module structure such that $L$ acts by zero
on $\omega_C$ and its action on $L$ is given by the structure map \eqref{theta-char-str-map}.
The derivation $\delta$ is $(d,\mathrm{id})$.  The distribution ${\mathcal D}$ is
the orthogonal complement of $\mathrm{Ker}(\bar\delta)$ where
$\bar\delta\colon\Omega^1_{X/S}\to \omega_{X/S}$
is the map of $\mathcal O_X$-modules associated with $\delta$.

\begin{lemma}
Let $\theta$ be a local nowhere vanishing section of $L$, and let $\alpha=\theta^2\in\omega_{C/S}$
be the image of $\theta\otimes \theta$ under the structure map \eqref{theta-char-str-map}. Then
\[
{\mathcal D}=\mathrm{Ker}(\alpha-\theta d\theta),
\]
where we use the natural embedding of $\omega_{C/S}=\Omega_{C/S}$ into $\Omega_{X/S}$.
\end{lemma}
\begin{proof}
The section $\alpha-\theta d\theta$ of $\Omega_{X/S}$ is nowhere vanishing.
We need to check that it belongs to $\mathrm{Ker}(\bar\delta)$. By definition $\bar\delta$ is the map of
$\mathcal O_X$-modules sending an exact form $d(f\oplus \psi)$ to
$df\oplus \psi$. Thus, $\bar\delta(\theta d\theta)=\theta^2$. 
On the other hand, since $\alpha\in \Omega_{C/S}$, we have 
$\bar\delta\alpha=\alpha$, and our assertion follows.
\end{proof}   

\subsection{Local automorphisms}
Let $X_0\to S_0$ be a smooth supercurve over an even reduced superscheme 
$S_0$. Let $S_0\hookrightarrow S$ be a closed embedding into a
superscheme corresponding to the nilpotent ideal $I_S$ of
$\mathcal O_S$ (so that $S_0=S_{\mathrm{red}}$).  Let
$\pi\colon X\to S$ be a smooth supercurve such that
$X|_{S_0}=X\times_{S}S_0\cong X_0$. The $S$-automorphisms of the
supercurve $X$ that induce the identity on $X_0$ act as the identity
on the underlying topological space of $X$. They are thus sections of
a sheaf $\mathrm{Aut}^0_{X/S}$ of unipotent groups. The Lie algebra
$\mathfrak g_{X/S}$ of $\mathrm{Aut}^0_{X/S}$ consists of even
superconformal vector fields vanishing on $X_0$.  It is a sheaf of
nilpotent graded Lie algebras on $X$ over $\pi^{-1}\mathcal O_S$.

\subsection{Extensions along odd moduli}\label{ss-odd moduli}
Suppose again that a smooth supercurve $X_0\to S_0$ over an even
reduced superscheme is given. Up to isomorphism this is the same as giving
an ordinary curve over $S_0$ with a theta-characteristic. For each $S$
with $S_{\mathrm{red}}=S_0$ we ask about
smooth supercurves $X\to S$ restricting to $X_0$ over $S_0$.

\begin{definition} 
  An {\it extension of $X_0\to S_0$ along $S$ } is a supercurve $X$ over $S$
  equipped with 
  an isomorphism
  $X\times_{S_0}S\cong X_0$ as supercurves over $S_0$.  Two extensions
  $X,X'$ along $S$ are isomorphic if there is an isomorphism $X\to X'$
  of supercurves over $S$ inducing the identity on $X_0$.
\end{definition}  

Working locally on smooth $S$, we may assume that $S=(|S_0|,\mathcal O_S)$
with $\mathcal O_S=\mathcal O_{S_0}\otimes_{\mathbb C} \bigwedge W$
for some finite dimensional vector space $W$.  Then we get an
extension $X\to S$, called the trivial extension,
by setting $X=(|X_0|,\mathcal O_X)$ and
$\mathcal O_X=\mathcal O_{X_0}\otimes\bigwedge W$. Any other extension
along $S$ is obtained from the trivial extension by twisting by a
$1$-cocycle.

\begin{lemma}\label{l-charclass} (cf. \cite[Lem.\ 2.3]{LeBrunRothstein1988})
  Let $X$ be an extension of $X_0$ along $S$. Then the isomorphism
  classes of extensions of $X_0$ along $S$ are in natural one to one
  correspondence with the first non-abelian cohomology
  $H^1(X,\operatorname{Aut}^0_{X/S})$.
\end{lemma}
\begin{proof}
  All extensions are locally isomorphic. Thus, every point $x\in |X_0|$
  has a neighbourhood $U\subset |X_0|$ so that any extension of
  $(U,\mathcal O_{X_0}|_U)$ along $S$ is isomorphic to
  $(U,\mathcal O_X|_{U})$. Hence, for any extension $X'$ we can find an open
  covering $\mathcal U=(|U_i|)$ of $|X_0|$ with corresponding
  supercurves $U_i=(|U_i|,\mathcal O_X|_{|U_i|})$, and isomorphisms
  $\varphi_i\colon U_i\to U_i'$. On non-empty intersections
  $|U_i|\cap |U_j|$ we get automorphisms
  $\varphi_{ij}=\varphi_{j}^{-1}\circ
  \varphi_i\in\mathrm{Aut}^0_{X/S}(U_i\cap U_j)$ defining a
  $1$-cocycle with values in $\mathrm{Aut}^0_{X/S}$. If $\varphi_i'$
  are a different choice of local isomorphisms then
  $\varphi_i'=\varphi_i\circ f_i$ with $f_i\in\mathrm{Aut}_{X/S}^0$
  and $\varphi_{ij}$ is replaced by the equivalent cocycle
  $f_j^{-1}\circ \varphi_{ij}\circ f_i$. Thus, to every extension $X'$
  we associate a well-defined characteristic class
  $[\varphi_{ij}]\in H^1(X,\operatorname{Aut}_{X/S}^0)$. Conversely,
  given a class in $H^1(X,\operatorname{Aut}_{X/S}^0)$ represented by
  a \v Cech cocycle $(\varphi_{ij})$ we obtain an extension $X'$ by
  gluing $\sqcup U_i$ over $|U_i|\cap |U_j|$ by the isomorphisms
  $\varphi_{ij}$. By construction, the characteristic class of $X'$ is
  $[\varphi_{ij}]$.
\end{proof}

\subsection{Local description of $\operatorname{Aut}^0_{X/S}$}
We give a local description of $\operatorname{Aut}^0_{X/S}$
(cf.~\cite{CraneRabin1988} and \cite[Section 2.1.1]{Witten}).
We can choose local fibre coordinates $z,\theta$ on an open
set $U\sub X$ so that the distribution ${\mathcal D}$ defining the supercurve is
spanned by $\partial_\theta+\theta\partial_z$ or, equivalently, is the
kernel of $dz-\theta d\theta$. 

We are going to construct elements of $\Gamma(U,\operatorname{Aut}^0_{X/S})$ of two types.  
For $f=f(z)\in\mathcal O_X^+(U)$, an even function such that
$\partial_\theta f=0$ and $f\equiv z\mod I_S$, we define a
superconformal map $S_f$ by
\begin{equation}\label{e-DefS}
z\mapsto S_f^*z=f(z),\quad \theta\mapsto S_f^*\theta=\sqrt{f'(z)}\,\theta,
\end{equation}
where $f'=\partial_zf$.
Note that the square root is well-defined since $f'(z)=1+$nilpotent.
On the other hand, for $\varphi=\varphi(z)\in \mathcal O_X^-(U)$, an odd function 
such that $\partial_\theta\varphi=0$, we define $T_\varphi$ by
\begin{equation}\label{e-DefT}
z\mapsto T_\varphi^*z=z+\theta\varphi(z),\quad 
\theta\mapsto T_\varphi^*\theta=\theta+\varphi(z)+\textstyle{\frac12}\,\theta\,\varphi(z)\,\varphi'(z).
\end{equation}

\begin{lemma}
The maps $S_f,T_\varphi$ are in $\Gamma(U,\operatorname{Aut}^0_{X/S})$
and every element of this group can be uniquely written in the form $S_f\circ T_\varphi$.
\end{lemma}

\begin{proof} A superconformal transformation is an automorphism such
  that $g^*(dz-\theta d\theta)=\lambda(dz-\theta d\theta)$ for some
  function $\lambda$.  Let $g^*z=g_0(z)+\theta g_1(z)$ and
  $g^*\theta=h_0(z)+\theta h_1(z)$, with $g_0,h_1$ even and $g_1,h_0$
  odd. The condition that $g$ reduces to the identity on $X_0$ is
  $g_0(z)=z+$nilpotent and $h_1(z)=1+$nilpotent.  Thus $g$ is
  superconformal if and only if
  \[
    h_0(z)=\frac{g_1(z)}{h_1(z)}, \quad
    h_1(z)^2=g'_0(z)+\frac{g_1(z)g_1'(z)}{g_0'(z)}.
  \]
  In this case $\lambda=g_0'(z)+h_0(z)h_0'(z)+2\theta h_0(z)h_1'(z)$.
  Since $g_0'(z)=1+$nilpotent, we can take the square root:
  \[
    h_1(z)=\pm\sqrt{g_0'(z)}\left(1+\frac{g_1(z)g_1'(z)}{2g'_0(z)^2}\right)
  \]
  The condition that $g$ restricts to the identity on $X_0$ selects
  the positive sign.  This gives
  \[
    g^*z=g_0(z)+\theta g_1(z),\quad g^*\theta=\sqrt{g_0'(z)}
    \left(\frac{g_1(z)}{g_0'(z)}
      +\theta\left(1+\frac{g_1(z)g_1'(z)}{2g'_0(z)^2}\right)\right).
  \]
  If we set $g_0(z)=f(z),g_1(z)=f'(z)\varphi(z)$, the formulas
  simplify slightly:
  \[
    g^*z=f(z)+\theta f'(z)\varphi(z)\quad
    g^*\theta=\sqrt{f'(z)}\left(\varphi(z)+\theta(1+\textstyle{\frac12}
      \varphi(z)\varphi'(z))\right).
  \]
  This reduces to $S_f^*$ for $\varphi=0$ and to $T^*_\varphi$ for
  $f(z)=z$. For general $f,\varphi$ we get
  $g^*=T_\varphi^*\circ S_f^*$.
\end{proof}
\begin{remark}
  The automorphism $T_\varphi$ is the exponential of an odd
  superconformal vector field (an odd vector field such that
  $[v,\mathcal D]\subset \mathcal D$]). Namely
  $T_\varphi=\exp(v(\varphi))$ with
  \[
    v(\varphi)
    =\theta\varphi\partial_z+\varphi\partial_\theta.
  \]
\end{remark}
\subsection{Coordinate free formulation}\label{s-coord}
The local generators of the sheaf of groups $\mathrm{Aut}^0_{X/S}$ of
fibre automorphisms restricting to the identity on $X_{\mathrm{red}}$
are described above using local coordinates, but we only need a
local splitting of $X$ to define them. Let
$X|_U\cong(U,\mathcal O_C\oplus L)$ be a splitting of $X\to S$ on some
open set $U$.  Here $C$ is an ordinary curve over $S$ with theta
characteristic $L$ and a section $f\oplus\psi$ with $\psi=\theta g$
for a local basis $\theta$ of the $\mathcal O_C$ module $L$
corresponds to $f(z)+\theta g(z)$ in the local coordinate
description. Then it is clear how to interpret \eqref{e-DefS},
\eqref{e-DefT} except possibly for the term $\varphi(z)\varphi'(z)$ in
\eqref{e-DefT}:
\begin{itemize}
\item $f$ is a local automorphism of $C$ which is the identity modulo
  the nilpotent ideal $I_S$ of $S$ and $S^*_f$ is its natural action
  on $\mathcal O_C\oplus L$. It preserves the splitting.
\item $\varphi$ is an odd section of $L^{-1}$ and $T^*_\varphi$ acts
  on $\mathcal O_X\oplus L$ as
  \[
    g\oplus\psi\to (g+\varphi\,\psi)\oplus
    (\psi+\varphi\,dg+\textstyle{\frac12Q}(\varphi)\psi),
  \]
  for a quadratic form $Q\colon L^{-1}\to \mathcal O_C$ with local
  coordinate expression $\varphi\varphi'$ to be described below. The
  product is in the algebra $\oplus_{j\in\mathbb Z} L^j$: for example
  $\varphi\,dg$ is the product of the section $\varphi$ of $ L^{-1}$
  and $dg$ of $\omega_C\cong L^2$; the result is a section of $L$.
\end{itemize}

It remains to describe the quadratic form $Q$.

  \begin{lemma} \label{l-w} Let $L$ be a theta-characteristic on $C\to S$.
    There is a unique map of sheaves of vector spaces
    \[
    w\colon L^{-1}\otimes_{\mathbb C}L^{-1}\to \mathcal O_C
    \]
    such that
    \begin{enumerate}
    \item $w(a,b)=-(-1)^{p(a)\,p(b)}w(a,b)$,
    \item $w(fa,b)=fw(a,b)+\langle d_{C/S}f, a\otimes b\rangle$,
    \end{enumerate}
    for all local sections $a,b\in L^{-1}$, of parity $p(a),p(b)$ and
    $f\in \mathcal O_C$. Here
    $\langle\;,\;\rangle$ is the canonical pairing
    $\omega_C\otimes (L^{-1}\otimes L^{-1})\to \mathcal O_C$ induced
    by the structure map $L^2\to\omega_{C/S}$.
  \end{lemma}
  
  \begin{proof}
    Let $e$ be a local basis of $L^{-1}$. Then, for any
    $f,g\in\mathcal O_C$,
    $w(fe,ge)=\langle(df\,g-f dg),e\otimes e\rangle$. This shows
    uniqueness. It remains to show that this formula defines a
    well-defined map $w$, independent of the choice of $e$. This
    follows from the identity
    \[
      d(hf)\,hg-hg\,d(hf)=h^2(fdg-gdf),\quad f,g,h\in\mathcal O_C.
    \]  
  \end{proof}
  \begin{remark} Locally,
    \[
      w(f(z)dz^{-1/2},g(z)dz^{-1/2})=f'(z)g(z)-f(z)g'(z)
    \]
    and Lemma \ref{l-w} implies that this expression is well-defined,
    independently of the choice of local coordinate $z$.
  \end{remark}
  The quadratic $Q\colon L^{-1}\to\mathcal O_C$ appearing in $T_\varphi$
  is the quadratic form
  associated to this bilinear form:
  \[
    Q(a)=\frac 12 w(a,a).
  \]

\subsection{Gluing construction of versal families in odd directions}

Let us start with any family $C\to S_0$ of curves with a relative theta-characteristic $L$ and a section
$p:S_0\to C$. This family defines a split
supercurve $X_0=(C,\mathcal O_C\oplus L)$ over $S_0$.
We want to extend $X_0$ to a supercurve $X\to S$ where
$\mathcal O_S=\mathcal O_{S_0}\otimes\bigwedge W$ with $W=H^1(C_s,L_s^{-1})^\vee$
for some $s\in S_0$
as in Section \ref{ss-odd moduli}. 

For this we assume in addition that $S_0$ is affine, $h^0(C_s,L_s^{-1})=0$ for all $s\in S_0$,
and the locally free $\OO(S_0)$-module $H^1(C,L^{-1})$ is trivialized:
$$H^1(C,L^{-1})\simeq W^\vee\ot \OO_{S_0}.$$
The required supercurve over $S$ is obtained from the trivial
extension $(|X_0|,\mathcal O_{X_0}\otimes\bigwedge W)$ by twisting by
a class in $H^1(X_0,\mathrm{Aut}^0_{X_0/S})$.
To define this class we cover $C$ by two open sets: an affine neighbourhood $U_0$ of 
$p(S_0)\sub C$, and $U_1=C\smallsetminus p(S_0)$. Then a class in
$H^1(X_0,\mathrm{Aut}^0_{X_0/S})$ is represented by a cocycle in
$\Gamma(U_0\cap U_1,\operatorname{Aut}^0_{X_0/S})$ which we take to be
\[
  T_\varphi=\exp(\varphi),\quad \varphi=\sum\eta_ib_i,
\]
with $b_i\in\Gamma(U_0\cap U_1,L^{-1})$ representing a basis of
$W^\vee\sub H^1(C,L^{-1})$ and the odd coordinates $(\eta_i)$ forming the dual basis of $W$.
Then the class $[b_i]\in H^1(C,L^{-1})$ is the Kodaira-Spencer class
corresponding to the vector field $\partial/\partial\eta_i$.

\subsection{Hodge complex and duality}\label{hodge-duality-sec}
Recall that the Hodge bundle of an ordinary smooth (proper, flat) family of curves
$C\to S$ of genus $g$ is $\pi_*\omega_{C/S}$. It is locally free of rank
$g$. The situation is more subtle for supercurves: consider for
example a supercurve over a point with theta-characteristic $L$ such
that $\operatorname{dim}\,H^0(L)=m>0$.  
Then $\pi_*\omega_{X}=H^0(\omega_X)=H^0(\omega_C)\oplus H^0(L)$ has rank
$g|m$, which varies from one supercurve to another. 
A more adequate object to work with is
the {\it Hodge complex} of a supercurve $X\to S$ defined as the derived direct image
  $R\pi_*\omega_{X/S}$ of its relative dualizing sheaf.

Below we will show that over the locus where the underlying theta-characterstic has no nonzero sections,
the sheaf $\pi_*\omega_{X/S}$
  is locally free of rank $g|0$. On the other hand, we will show that for the universal family, the sheaves
  $\pi_*\omega_{X/S}$ and $R^1\pi_*\omega_{X/S}$ are not locally free near every point where $h^0(L)\neq 0$ (see Theorem 
  \ref{non-loc-free-thm}).

Recall that for a family of supercurves $\pi:X\to S$ one has a canonical
map 
$$\tau:R^1\pi_*\om_{X/S}\to \OO_S$$
(see \eqref{Ber-curve-trace-map}). Since $R\pi_*\om_{X/S}$ has cohomology in degrees $0$ and $1$ we can
view this map as a morphism in the derived category
$$R\pi_*\om_{X/S}\to \OO_S[-1].$$

Assume that $S$ is affine and the corresponding reduced family
$C\to S_0$ has a marked point $\ov{p}:S_0\to C$,
and let $U_0,U_1$ be an open covering of $C$, such that $U_0$ is an affine
neighborhood of $\ov{p}(S_0)$ and $U_1$ is the complement to $\ov{p}(S_0)$.
Then we can calculate $R^1\pi_*\om_{X/S}(S)$ as 
$$\coker(\om_{X/S}(U_0)\oplus \om_{X/S}(U_1)\to \om_{X/S}(U_{01})).$$
Thus, we get a canonical (even) residue map
$$\Res_{\ov{p}}:\om_{X/S}(U_{01})\to \OO_S,$$
which vanishes on the image of $\om_{X/S}(U_0)$.
Furthermore, since $\tau$ vanishes on the image of $\de:R^1\pi_*\om_{X/S}$,
it follows that $\Res_{\ov{p}}$ vanishes on $\de(\om_{X/S}(U_{01}))$.

We can also replace $U_{01}$ be a formal punctured disk around $\ov{p}$.

\begin{remark}\label{superdisk-remark}
Let us consider the (split) formal superdisk $D$ and the punctured formal superdisk $D'$ over $S$, given by the algebras
$$\OO_D=\OO_S[\![z]\!]\oplus \OO_S[\![z]\!]\th,$$
$$\OO_{D'}=\OO_S(\!(z)\!)\oplus \OO_S(\!(z)\!)\th,$$
where we think of $\th$ as a formal square root $dz^{1/2}$.
We have 
$$\om_D=[dz|d\th]\OO_D, \ \ \om_{D'}=[dz|d\th]\OO_{D'},$$
and the derivation $\de$ is given by 
$$\de:\OO_D\to \om_D, \OO_{D'}\to \om_{D'}: f(z)+g(z)\th\mapsto [dz|d\th](g(z)+f'(z)\th).$$
Then the map 
$$\Res:[dz|d\th]\OO_{D'}\to \OO_S$$
is the composition of the projection to $[dz|d\th]\OO_S(\!(z)\!)\th$ with the usual residue map on $\OO_S(\!(z)\!)$. 
Note that the fact that $\Res$ vanishes on $[dz|d\th]\OO_S(\!(z)\!)$ is a consequence of the fact that it vanishes on the image of $\de$. 
\end{remark}

There is an analogue of Grothendieck-Serre duality for families of supercurves in which $\om_{X/S}$
plays a role of a relative dualizing sheaf, see \cite[Sec.\ 2]{VMP} (see also \cite{RR}, where the case of
Serre duality on supercurves is worked out in detail). 
This duality gives an isomorphism in the derived category of $S$,
$$R\pi_*(V)^\vee\simeq R\pi_*(V^\vee\ot \om_{X/S})[1],$$
for any perfect complex $V$ over $X$.
A part of this theory is the trace
map 
$$\tau':R\pi_*\om_{X/S}\to \OO_S[-1].$$
We claim that $\tau'$ coincides with $\tau$ up to a sign. Indeed, this follows from the compatibility of $\tau'$
with the residue map: if $X/S$ is equipped with a relative divisor $p\sub X$, where $p\to S$ is of relative dimension $0|1$, 
then we have an exact triangle
$$R\pi_*\om_{X/S}\to R\pi_*\om_{X/S}(p)\to \om_{X/S}(p)|_p\rTo{\a} R\pi_*\om_{X/S}[1]$$
such that the composition of $\a$ with $\tau'$ is the residue map
$\om_{X/S}(p)|_p\to \OO_S$.

\section{Local freeness and non-freeness}\label{loc-free-sec}

\subsection{Base change formalism}

Let $f:X\to S$ be a projective map of supervarieties, and let $\FF$ be a coherent sheaf on $X$, flat over $S$.

Assume also that $S$ is affine, $S=\Spec(A)$, where $A$ is a supercommutative ring (finitely generated over a field).
Then exactly as in the purely even case (see e.g., \cite[Prop.\ III.12.2]{Hart}) one proves that 
there exists a bounded above complex of projective $A$-modules $Q^\bullet$, such that
for every finitely generated $A$-module $M$, one has a functorial isomorphism
$$H^i(X,\FF\otimes_A M)\simeq H^i(Q^\bullet\otimes_A M)$$
The study of these cohomology functors as $M\mapsto T^i(M)$
is parallel to the purely even case (see \cite{BHRP} for details).

Assume in addition that $A$ is local with maximal ideal $\mathfrak m$.
For every $A$-module $M$ there is a natural map
$$\phi^i(M):T^i(A)\otimes M\to T^i(M).$$
Let $s$ be the closed point of $S$ with residue field $k(s)=A/\mathfrak m$.

\begin{theorem}\label{coh-bc-thm} (see \cite[Sec.\ III.12]{Hart})
(a) If $\phi^i(k(s))$ is surjective then $\phi^i(M)$ is an isomorphism for any $M$.

\noindent
(b) Assume that $\phi^i(k(s))$ is surjective. Then $\phi^{i-1}(k(s))$ is surjective if and only if $T^i(A)$ is a free $A$-module.
\end{theorem}

Note that $T^i(A)\simeq H^0(S,R^if_*\FF)$, so we can use the above Theorem to check local freeness of $R^if_*\FF$.


\subsection{Local freeness for a family of supercurves}\label{loc-free-super-sec}

Now assume that the fiber dimension of the map between reduced (usual) schemes induced by $f$ is $1$.
Then $T^{>1}=0$, so by Theorem \ref{coh-bc-thm} applied to $i=2$, $\phi^1(k(s))$ is an isomorphism.
Thus, by Theorem \ref{coh-bc-thm}(b) applied to $i=1$, if $\phi^0(k(s))$ is surjective then $R^1f_*(\FF)$ is locally free.
Since $\phi^{-1}(k(s))$ is always surjective, under the same assumption $f_*\FF$ is also locally free.

\begin{prop}\label{loc-free-prop}
Let $\pi:X\to S$ be a family of supercurves, and let $\ov{\pi}:C\to S_0$ be the corresponding usual family of
curves with a relative theta-characteristic $L$. Assume that $R\ov{\pi}_*L=0$ (or equivalently, $H^*(L_s)=0$ on every fiber).
Then $R^1\pi_*\OO_X$, $\pi_*\om_{X/S}$ and $R^1\pi_*\om_{X/S}$ are locally free, and the natural map
$$\OO_S\to \pi_*\OO_X$$
is an isomorphism.
\end{prop}

\begin{proof}
First, let us apply Theorem \ref{coh-bc-thm} to the sheaf $\OO_X$.
The assumption on $L$ implies that for every point $s\in S$ the map 
$$\phi^0(k(s)):\pi_*\OO_X\otimes k(s)\to H^0(C_s,\OO)$$ 
is surjective. Hence, $\pi_*\OO_X$ and $R^1\pi_*\OO_X$ are locally free,
and the map $\OO_S\to\pi_*\OO_X$ is an isomorphism. 
This implies that locally we have a splitting $R\pi_*\OO_X\simeq \pi_*\OO_X\oplus R^1\pi_*\OO_X[-1]$.

By Grothendieck duality, we have 
$$R\pi_*\om_{X/S}\simeq (R\pi_*\OO_X)^\vee[-1]\simeq (\pi_*\OO_X\oplus R^1\pi_*\OO_X[-1])^\vee[-1].$$
It follows that $R^1\pi_*\om_{X/S}\simeq (\pi_*\OO_C)^\vee$ and
$\pi_*\om_{X/S}\simeq (R^1\pi_*\OO_C)^\vee$ are locally free.
\end{proof}

\subsection{Local non-freeness}

The following result implies that $R^i\pi_*\omega_{X/S}$ is not locally free in the vicinity of any point of the moduli space
where the corresponding theta-characteristic
has non-trivial sections. In particular, for the universal family $\pi:X\to \SS^-$ over the component
corresponding to odd theta-characteristics, the sheaves $\pi_*\om_{X/S}$ and $R^1\pi_*\om_{X/S}$
are not locally free.

\begin{theorem}\label{non-loc-free-thm}
Let $\pi:X\to S$ be the universal family of supercurves over the moduli stack of supercurves.
Assume that the corresponding theta-characteristic $L_s$ on $C_s$ has $h^0(C_s,L_s)\neq 0$ for some
point $s\in S$.
Then the sheaves $\pi_*\om_{X/S}$ and $R^1\pi_*\om_{X/S}$ are not locally free near $s$. 
\end{theorem}

We will use the following general obstacle to local freeness.

\begin{lemma}\label{local-non-freeness-lem} 
Let $Q^\bullet=[Q^0\rTo{f} Q^1]$ be a two-term complex of vector bundles over a superscheme $S$.
For every point $s\in S$ let us define the map
$$\kappa(s)=\kappa_{Q^\bullet}(s):T_sS\otimes H^0(Q^\bullet|_s)\to H^1(Q^\bullet|_s),$$
where $T_sS$ is the Zariski tangent space to $S$ at $s$ as follows.
Let $\fm$ be the maximal ideal in the local ring of $s$. Then we have
$$H^0(Q^\bullet|_s)=\ker (Q^0/\fm Q^0\to Q^1/\fm Q^1),$$
so we have a natural map
\begin{equation}\label{obstr-kappa-definition-eq}
H^0(Q^\bullet|_s)\to \fm Q^1/(\fm f(Q^0)+\fm^2 Q^1)\simeq \fm/\fm^2\ot (Q^1/f(Q^0)),
\end{equation}
from which $\kappa$ is obtained by dualization.

\noindent
(i) The map $\kappa(s)$ depends only on a quasi-isomorphism class of $Q^\bullet$.

\noindent
(ii) Locally near $s$ there exists a complex 
$$H^0(Q^\bullet|_s)\ot \OO_S \rTo{g} H^1(Q^\bullet|_s)\ot \OO_S,$$ 
quasi-isomorphic to $Q^\bullet$,
such that the entries of $g$ are in $\fm$ and $g\mod \fm^2$ corresponds to $\kappa(s)$.

\noindent
(iii) If $\kappa(s)$ is nontrivial then $\und{H}^1 Q^\bullet$ is not locally free near $s$.
\end{lemma}

\begin{proof}
(i) This follows from the naturality of the map \eqref{obstr-kappa-definition-eq}
with respect to chain maps between two-term complexes.

\noindent
(ii) We can find bases of $Q^0|_s$ and $Q^1|_s$ such that the matrix of $f\mod \fm$ has form
$$f\mod \fm=\left(\begin{matrix} 1_r & 0 \\ 0 & 0\end{matrix}\right).$$
Lifting these bases to local bases of $Q^0$ and $Q^1$ and then adjusting them appropriately, we can
arrange the matrix of $f$ to be of the form
$$f=\left(\begin{matrix} 1_r & 0 \\ 0 & M\end{matrix}\right)$$
where $M$ has entries in $\fm$. But this means that there exists a quasi-isomorphic two-term complex 
with the differential given by $M$, and the assertion follows.

\noindent
(iii) If $\und{H}^1Q^\bullet$ is locally free near $s$ then in a neighborhood of $s$ there exists a 
two-term complex quasi-isomorphic to $Q^\bullet$ with zero differential. Hence, by (i), $\kappa(s)=0$. 
\end{proof}


The crucial part of the proof of Theorem \ref{non-loc-free-thm} is the following calculation of
a component of the obstruction map.

\begin{lemma}\label{loc-free-obstr-calc-lem} 
Let $[Q^0\to Q^1]$ be a resolution of $R\pi_*\om_{X/S}$ in a neighborhood of a point $s$ of the moduli space $S$. Then 
we have decompositions into even and odd components,
$$H^i(\om_{X_s})=H^i(\om_{C_s})\oplus H^i(L_s)$$
and the component
\begin{equation}\label{L-cup-product-map-eq}
H^1(L_s^{-1})\ot H^0(\om_{C_s})\to H^1(L_s)
\end{equation}
of $\kappa(s)$ is given by the cup-product. Furthermore,
the map \eqref{L-cup-product-map-eq} is surjective.
\end{lemma}

\begin{proof}
It is enough to consider the family $\pi:X\to S$ over $S=\Spec(\C[\eps]/\eps^2)$ (where $\eps$ is odd) deforming the supercurve $X_0$, with
$\OO_{X_0}=\OO_C\oplus L$ over $\Spec(\C)$, where $C=C_s$, $L=L_s$, associated with an odd tangent vector 
$v\in H^1(C,L^{-1})$.

Let $C=U_0\cup U_1$ be an open affine covering of $C$. Then $R\pi_*\om_{X/S}$
is represented by the complex
$$[\omega_{X/S}(U_1)\to \omega_{X/S}(U_{01})/\omega_{X/S}(U_0)].$$
Furthermore, we have an identification of $\OO_X(U_i)$ with $\OO_{X_0}(U_i)[\eps]$, so that the two identifications of
$\OO_X(U_{01})$ with $\OO_{X_0}(U_{01})[\eps]$ differ by an automorphism $T$ given by $\exp(\eps \wt{v})$,
where $\wt{v}$ is an odd vector field over $U_{01}$ representing the class $v\in H^1(C,L^{-1})$.
Thus, we can represent $R\pi_*\om_{X/S}$ by the complex
\begin{equation}\label{Hodge-complex-T-differential}
[\om_{X_0/S}(U_1)[\eps]\rTo{\alpha\mapsto T(\alpha|_{U_{01}})} \om_{X_0/S}(U_{01})/\om_{X_0/S}(U_0)].
\end{equation}

We claim that the restriction of the induced action of $T$ on the Berezinian of $\OO_{X_0}(U_{01})[\eps]$ over 
$\C[\eps]/\eps^2$ to $\om_{C}(U_{01})$ is given by
$$\om_{C}(U_{01})\to \om_{C}(U_{01})[\eps]\oplus L(U_{01})[\eps]: \alpha\mapsto \alpha+\eps\lan \wt{v},\alpha\ran.$$
Indeed, in general this action is given $\alpha\mapsto \alpha+L_{\wt{v}}(\alpha)$, where
$L_{\wt{v}}$ is the Lie derivative. Since $\wt{v}$ is an odd vector field, our claim follows easily from the description
of the differential $\delta$ in terms of the splitting $\OO_{X_0}=\OO_{C}\oplus L$ (see Sec.\ \ref{supercurves-even-base-sec}).

Thus, if $\a$ is a global section of $\om_C$ then the differential of the complex \eqref{Hodge-complex-T-differential}
applied to $\a|_{U_1}$ gives $\eps\lan \wt{v},\alpha\ran$, which leads to the required formula for $\kappa(s)$.

To check the surjectivity statement we observe that by Serre duality, the map \eqref{L-cup-product-map-eq} is dual
to the map
$$H^0(L_s)\to \Hom(H^0(\om_{C_s}),H^0(\om_{C_s}\ot L_s)):s\mapsto (\a\mapsto \a s),$$
which is clearly injective: for a nonzero $s\in H^0(L_s)$, the corresponding map $H^0(\om_{C_s})\to H^0(\om_{C_s}\ot L_s)$
is injective, hence nonzero.
\end{proof}

We also need the following algebraic result about free modules over an exterior algebra.

\begin{lemma}\label{exterior-alg-modules-lem} 
Let $R$ be a commutative ring, and let $\AA=\bigwedge^*_R(\bigoplus_{i=1}^n R\th_i)$ be the exterior algebra with $n$
generators $\th_1,\ldots,\th_n$ over $R$. Let us consider the ideal $\NN=\bigwedge^{\ge 1}_R(R^n)\sub \AA$. For an $\AA$-module
$M$, we denote by $M_\NN$ the annihilator of $\NN$ in $M$. Suppose we have a morphism $\phi:M\to M'$ of
free $\AA$-modules of finite rank, such that the induced morphism $M_\NN\to M'_\NN$ is an isomorphism.
Then $\phi$ is an isomorphism. 
\end{lemma}

\begin{proof} 
For a free $\AA$-module of finite rank $M$, the map induced by the $\AA$-action, 
$$\th_1\cdot\ldots\cdot\th_n\cdot:M/\NN M\to M_\NN$$
is an isomorphism. Since $\phi$ is compatible with the $\AA$-action, we deduce that the morphism
$$M/\NN M\to M'/\NN M'$$
is an isomorphism. Since $\NN$ is nilpotent, this implies that $\phi$ is surjective. Hence, $\phi$
is a projection onto a direct summand. Thus, if $K=\ker(\phi)$, we deduce that $K/\NN K=0$, so $K=0$.
\end{proof}

\begin{proof}[Proof of Theorem \ref{non-loc-free-thm}]
Set $C=C_s$, $L=L_s$.
By Serre duality, the cup-product map $H^1(C,L^{-1})\ot H^0(C,\om_C)\to H^1(C,L)$ corresponds
to the cup-product map 
\begin{equation}\label{omega-L-obstr-cup-product-eq}
H^0(C,\om_C)\ot H^0(C,L)\to H^0(C,\om\ot L).
\end{equation}
The latter map is always nonzero whenever $H^0(C,L)\neq 0$.
Thus, using Lemma \ref{loc-free-obstr-calc-lem}, we see that the obstruction $\kappa(s)$ associated with a
resolution of $R\pi_*\om_{X/S}$ is nonzero, so
by Lemma \ref{local-non-freeness-lem}, the sheaf $R^1\pi_*\om_{X/S}$ is not locally free near $s$.

Next, we will deal with the sheaf $\pi_*\om_{X/S}$, replacing $S$ by an affine neighborhood of $s$.
The trace map \eqref{Ber-curve-trace-map} can be viewed as a morphism in derived category
$$R\pi_*\om_{X/S}\to \OO_S[-1].$$
More precisely, if $R\pi_*\om_{X/S}$ is represented by a complex of vector bundles $[A\rTo{f} B]$ then $\tau$ corresponds
to a map $B/f(A)\to \OO_S$, so we can view it as a chain map $[A\to B]\to \OO_S[-1]$. 
Furthermore, since $\tau$ is surjective, there exists a splitting $\OO_S\to B$ of the projection $B\to \OO_S$,
which we can view as a chain map $\OO_S[-1]\to [A\to B]$. Thus, we get a splitting in derived category,
$$R\pi_*\om_{X/S}\simeq P^\bullet\oplus \OO_S[-1].$$
Next, by Lemma \ref{local-non-freeness-lem}(ii), locally near $s$,
the complex $P^\bullet$ is quasi-isomorphic to a two-term complex $Q^\bullet$ of the form
$$H^0(\om_C)\ot \OO_S\oplus H^0(L)\ot \OO_S\rTo{(f^-,f^+)} H^1(L)\ot \OO_S$$
(here the space $H^0(\om_C)$ is even, while the spaces $H^i(L)$ are odd),
where the linear terms of $f^-$ are given by the map \eqref{L-cup-product-map-eq}.

Let $S_0\sub S$ be the reduced subscheme. Then the restriction of $Q^\bullet$ to $S_0$ splits
into the direct sum of $H^0(\om_C)\ot\OO_{S_0}$ and the complex 
\begin{equation}\label{theta-char-complex-eq}
H^0(L)\ot \OO_{S_0}\rTo{f^+|_{S_0}} H^1(L)\ot \OO_{S_0}
\end{equation}
computing the direct image of the universal theta-characteristic.

Set $\FF=\und{H}^0Q^\bullet=\ker(f^-,f^+)$. We want to check that $\FF$ is not locally free.
We distinguish two cases.

\medskip

\noindent
{\bf Case $h^0(L)$ is even}. Then the complex \eqref{theta-char-complex-eq} is generically acyclic, so
$f^+|_{S_0}$ is generically an isomorphism. In particular, $f^+$ is injective and the sheaf $\coker(f^+)$ has rank $0$.
It follows that $\FF$ can be identified with the kernel of the morphism
$$H^0(\om_C)\ot \OO_S\to \coker(f^+)$$
induced by $f^-$. In particular, $\FF$ is an $\OO_S$-submodule of $H^0(\om_C)\ot\OO_S$.

Locally we have a splitting $\OO_S\simeq {\bigwedge}(W)\ot \OO_{S_0}$, with $\NN={\bigwedge}^{\ge 1}(W)\ot \OO_{S_0}$.
Assume that $\FF$ is a locally free $\OO_S$-module and let us lead this to a contradiction.
Consider the submodule $\FF_\NN$ consisting of
sections annihilated by $\NN$. 
Since $f^-$ has entries in $\NN$, we get that 
$$\FF_\NN=H^0(\om_C)\ot {\bigwedge}^{2g-2}(W)\ot\OO_{S_0}.$$
Hence, by Lemma \ref{exterior-alg-modules-lem}, $\FF=H^0(\om_C)\ot \OO_S$.
In other words, the image of $f^-:H^0(\om_C)\ot\OO_S\to H^1(L)\ot \OO_S$ 
is contained in the image of $f^+:H^0(L)\ot \OO_S\to H^1(L)\ot \OO_S$.
Let us consider the following component of $f^-$:
$$f^-_1:H^0(\om_C)\ot {\bigwedge}^{2g-3}(W)\ot\OO_{S_0}\to H^1(L)\ot {\bigwedge}^{2g-2}(W)\ot\OO_{S_0}.$$
Modulo the maximal ideal we get the linear map 
$$H^0(\om_C)\ot W^\vee\simeq H^0(\om_C)\ot H^1(L^{-1})\to H^1(L)$$
which is given by the cup-product and is surjective (see Lemma \ref{loc-free-obstr-calc-lem}). Hence, 
$H^1(L)\ot {\bigwedge}^{2g-2}(W)\ot\OO_{S_0}=(H^1(L)\ot\OO_S)_{\NN}$ is in the image of $f^-$, so it has to be in the image of $f^+$.
Since $f^+$ is injective, it can only be in the image of $(H^0(L)\ot \OO_S)_{\NN}$.
But the map
$$H^0(L)\ot \OO_{S_0}\simeq (H^0(L)\ot \OO_S)_{\NN}\to (H^1(L)\ot\OO_S)_{\NN}\simeq H^1(L)\ot\OO_{S_0}$$
induced by $f^+|_{S_0}$ has entries in the maximal ideal, so its image cannot be everything, which is a contradiction.

\medskip

\noindent
{\bf Case $h^0(L)$ is odd}. Since the locus of local freeness is open, it is enough to consider the case $h^0(L)=1$
(since the set of such points is dense in the odd component of the moduli space).
In this case $f^+|_{S_0}=0$, so $f^+$ has entries in $\NN^2$.
Hence, we have
$$\FF_\NN=H^0(\om_C)\ot\OO_{S_0}\oplus H^0(L)\ot \OO_{S_0},$$
which is a locally free $\OO_{S_0}$-module of rank $g+1$ (we count both odd and even
generators). Let us also consider $\FF_{\NN^2}$, the annihilator of $\NN^2$ in $\FF$.
We have
$$\FF_{\NN^2}=
\ker((H^0(\om_C)\ot\OO_S)_{\NN^2}\rTo{f^-} (H^1(L)\ot\OO_S)_\NN)\oplus (H^0(L)\ot\OO_S)_{\NN^2},$$
so that
$$\FF_{\NN^2}/\FF_\NN\simeq \ker(W^\vee\ot H^0(\om_C)\ot \OO_{S_0}\rTo{f^-} H^1(L)\ot \OO_{S_0})
\oplus W^\vee\ot\OO_{S_0}.$$
It follows that $\FF_{\NN^2}/\FF_\NN$ is locally free over $\OO_{S_0}$ of rank $(g+1)\dim(W)-1$.
But for a free $\OO_S$-module $\GG$, we would have 
$$\rk(\GG_{\NN^2}/\GG_\NN)=\dim(W)\cdot \rk(\GG_\NN).$$
Hence, $\FF$ is not locally free over $\OO_S$ near $s$.
\end{proof}


\section{Symplectic picture}\label{s-4}

\subsection{Support of a complex of bundles}

Let $Q^\bullet$ be a bounded complex of vector bundles over a superscheme $S$.

\begin{definition}
Let us define the quasicoherent sheaf of ideals $\II(Q^\bullet)\sub\OO_S$ as 
the kernel of the canonical morphism
$$\OO_S\to \und{H}^0((Q^\bullet)^\vee\ot Q^\bullet).$$
Equivalently, this ideal consists of $f\in\OO$ such that
the morphism $f\cdot \id$ of $Q^\bullet$ is locally homotopic to zero.
\end{definition}

We need some simple properties of this definition,
which in particular imply that $\II(Q^\bullet)$ defines a subscheme structure on the support of $Q^\bullet$. 

\begin{lemma}\label{complex-support-lem}
(i) If complexes $Q^\bullet$ and $R^\bullet$ are quasi-isomorphic then $\II(Q^\bullet)=\II(R^\bullet)$.
In particular, one has $\II(Q^\bullet)=\OO_S$ if and only if $Q^\bullet$ is acyclic.

\noindent
(ii) One has $\II(Q^\bullet)\sub \cap_i \Ann(\und{H}^i Q^\bullet)$.

\noindent 
(iii) For an open subset $U\sub S$ one has $\II(Q^\bullet)|_U=\II(Q^\bullet|_U)$.

\noindent
(iv) Consider a $2$-term complex $A\rTo{\de} B$, such that $\de$ is an isomorphism on a Zariski open subset $U\sub S$.
Assume that $S$ is affine, and the homomorphism $\OO(S)\to \OO(U)$ is injective.
Then $f\in H^0(S,\II(A\to B))$ if and only if the morphism $f|_U\de^{-1}:B|_U\to A|_U$ has a regular extension to $S$.  
\end{lemma}

\begin{proof}
Part (i) follows from the fact that any 
quasi-isomorphism between bounded complexes of vector bundles locally becomes a homotopy equivalence.
Parts (ii) and (iii) are straightforward. For part (iv), we observe that a homotopy $h:B\to A$ between
$f\cdot \id$ and $0$ satisfies $h\de=f\id_A$, which is equivalent to $h$ being an extension of $f\de^{-1}$.
\end{proof}

\subsection{Isotropic intersections/co-intersections}

Let $\VV, (\cdot,\cdot)$ be a symplectic vector bundle of rank $2m|2n$ over an irreducible superscheme $S$,
and let $L_1,L_2\sub\VV$ be a pair of maximal isotropic subbundles of rank $m|n$
(this means that $L_i$ is isotropic and the morphism $\VV/L_i\to L_i^\vee$ induced by
the pairing is an isomorphism).
Then we have a complex of bundles over $S$, concentrated in degrees $0$ and $1$,
$$C^\bullet(L_1,L_2)=C^\bullet(\VV;L_1,L_2): [L_1\to \VV/L_2]$$
that controls the behavior of the intersections $f^*L_1\cap f^*L_2=\und{H}^0 f^*C^\bullet$ for arbitrary morphisms
$f:T\to S$, and of the co-intersection $\VV/(L_1+L_2)\simeq \und{H}^1C^\bullet$.
Note that the symplectic form induces an isomorphism $\VV/L_2\simeq L_2^\vee$, so we can also write the above complex
as 
$$[L_1\to L_2^\vee],$$
which in particular shows that $C^\bullet(L_2,L_1)$ is dual to $C^\bullet(L_1,L_2)[-1]$.
On the other hand, we have a natural quasi-isomorphism to $C^\bullet(L_1,L_2)$ from
the complex
$$[L_1\oplus L_2\to \VV].$$
Thus, as an object in the derived category, it is self-dual up to a shift.

We say that $L_1$ and $L_2$ are transversal if the natural map $L_1\oplus L_2\to \VV$ is an isomorphism.
This is equivalent to exactness of the complex $C^\bullet(L_1,L_2)$.

\begin{prop}\label{selfdual-rep-prop}
(i) Assume that $L'_2\sub \VV$ is a maximal isotropic subbundle transversal to both $L_1$ and $L_2$.
Let us identify $L'_2$ with $L_2^\vee$ using the pairing between $L_2$ and $L'_2$ given by the symplectic form.
Then $L_1$ is the graph of a symmetric morphism
$\phi:L_2\to L'_2\simeq L_2^\vee$
(i.e., $\phi^*=\phi$), and the complex $C^\bullet(L_1,L_2)$ is isomorphic to
$$[L_2\rTo{\phi} L'_2].$$

\noindent
(ii) Assume that the complex $C^\bullet(L_1,L_2)$ is generically exact. 
If $n$ (the dimension of $L_i^-$)
is even then locally there exists a maximal isotropic subbundle transversal to both $L_1$ and $L_2$.
For any $n$ there exists a symmetric representative of $C^\bullet(L_1,L_2)$,
i.e., there exists a quasi-isomorphic complex of the form 
$$M\rTo{\phi} M^\vee$$
with $\phi^*=\phi$. Furthermore, if $n$ is even then we can find such $M$ of rank $m|n$ (same as the rank of $L_i$),
while for odd $n$ we can find $M$ of rank $m|(n+1)$.
\end{prop}

\begin{proof} 
(i) The first assertion is standard. Using the identification of $L_2^\vee$ with $L'_2$, we can identify $C^\bullet$ with
$[L_1\to L'_2]$, where the map is given by the projection along $L_2$.
Since the projection $L_1\to L_2$ along $L'_2$ is an isomorphism, 
we get an isomorphism of complexes
$$[L_1\to L'_2]\simeq [L_2\rTo{\phi} L'_2]$$
as claimed.

\noindent
(ii) Assume first that $n$ is even. By part (i), in this case it is enough to check that locally we can choose a maximal isotropic 
subbundle $L'_2\sub \VV$, transversal to both $L_1$ and $L_2$.
Indeed, it is easy to see that one can choose a maximal isotropic subbundle $\wt{L}_2\sub\VV$,
complementary to $L_2$. Next, for every point $x\in S$, we have 
$$\VV|_x=\VV^+_x\oplus \VV^-_x, \ \ L_i|_x=L_{ix}^+ \oplus L_{ix}^-,$$
where $L_{ix}^+$ are Lagrangian subspaces of the symplectic vector space $\VV^+|_x$ and
$L_{ix}^-$ are maximal isotropic subspaces of the orthogonal vector space $\VV^-|_x$.
Since generically $L_1|_x$ and $L_2|_x$ have trivial intersection in $\VV|_x$,
it follows that at any point the intersection $L^-_1|_x\cap L^-_2|_x$ has even dimension. 

Then an easy linear algebra argument (using the fact that $n-\dim L_{1x}^-\cap L_{2x}^-$ is even) 
shows that there exists a skew-symmetric map $\psi_x^-:\wt{L}^-_{2x}\to L_{2x}^-$ such that the graph
of $\psi_x^-$ has trivial intersection with $L_{1x}^-$. Similarly, there exists a symmetric map
$\psi_x^+:\wt{L}^+_{2x}\to L_{2x}^+$ such that the graph of $\psi_x^+$ has trivial intersection with
$L_{1x}^+$. Now we view $\psi_x=(\psi_x^+,\psi_x^-)$ as a symmetric map of vector spaces, 
extend it to a symmetric morphism of bundles $\psi:\wt{L}_2\to L_2$ in a neighborhood of $x$ and 
define $L'_2$ to be the graph of $\psi$.
 
Now let us consider the case when $n$ is odd. Then we can set $\VV'=\VV\oplus (\OO e_1\oplus \OO e_2)$, 
where $e_1$ and $e_2$ are odd, and equip $\VV'$
with a symplectic pairing using the pairing on $\VV$ and the standard symmetric pairing on $\OO e_1\oplus \OO e_2$, so that
$e_1$ and $e_2$ are isotropic and $(e_1,e_2)=1$.
Let us also set
$$L'_1=L_1\oplus \OO e_1, \ \ L'_2=L_2\oplus \OO e_2.$$
Then $L'_1$ and $L'_2$ are maximal isotropic in $\VV'$, and
the complex $L_1\to \VV/L_2$ is quasi-isomorphic to 
$$[L_1\oplus \OO e_1\to \VV/L_2\oplus \OO e_1]\simeq [L'_1\to \VV'/L'_2].$$
It remains to apply the case of even $n$.
\end{proof}

We have the following nice reduction property of the complexes $C^\bullet(\VV;L_1,L_2)$.

\begin{lemma}\label{reduction-lem} 
Let $L_1,L_2\sub \VV$ be a pair of maximal isotropic subbundles in a symplectic bundle of rank $2m|2n$.
Let $M\sub L_1$ be a subbundle such that the map $L_2\to M^\vee$, induced by the pairing, is surjective (equivalently,
$\VV=L_2+M^\perp$).
Then the bundle $\ov{\VV}:=M^{\perp}/M$ has an induced symplectic structure and induced maximal isotropic subbundles
$\ov{L}_1=L_1/M$, $\ov{L}_2=L_2\cap M^\perp$ in $\ov{\VV}$. Furthermore, there is a natural quasi-isomorphism of
complexes
$$C^\bullet(\VV;L_1,L_2)\to C^\bullet(\ov{\VV};\ov{L}_1,\ov{L}_2).$$
\end{lemma}

\begin{proof} 
Note that $\ov{L}_2$ embeds into $\ov{\VV}$ since $L_2\cap M\sub (L_2+M^\perp)^\perp=0$.
One can easily check that $\ov{L}_1$ and $\ov{L}_2$ are maximal isotropic in $\ov{\VV}$.
Now we use a quasi-isomorphism of complexes
$$[L_1\to \VV/L_2]\to [L_1/M\to \VV/(M+L_2)],$$
and observe that the natural map
$$\ov{\VV}/\ov{L}_2\simeq M^{\perp}/(M+L_2\cap M^\perp)\to \VV/(M+L_2)$$
is an isomorphism.
\end{proof}

\subsection{Determinants}\label{det-sec}

Recall that when a complex of vector bundles of even rank, concentrated in degrees $0$ and $1$,
$C^\bullet=[C^0\rTo{d} C^1]$, is generically exact then we have a natural section
$$\th_{C^\bullet}=\det(d)\in \Det(C^\bullet)^{-1}=\Det(C_0)^{-1}\ot \Det(C_1),$$
which vanishes precisely on the locus where $C^\bullet$ fails to be exact. 
Similarly, for an arbitrary complex of (super)vector bundles $C^\bullet=[C^0\rTo{d} C^1]$,
exact over an open subset $U$, there is a natural invertible section
$$\th=\th_{C^\bullet}=\ber(d)\in H^0(U,\Ber(C^\bullet)^{-1}).$$
The key property is that a quasi-isomorphism of complexes $C^\bullet\to D^\bullet$ induces an isomorphism
$\Ber(C^\bullet)\rTo{\sim} \Ber(D^\bullet)$ such that $\th_{C^\bullet}$ gets identified with $\th_{D^\bullet}$.

Now let $L_1,L_2\sub \VV$ be a pair of maximal isotropic subbundles in a symplectic bundle.
Assume that $L_1$ and $L_2$ are transversal over a nonempty open $U\sub S$.
Then we define
$$\th(L_1,L_2)\in \Ber(C^\bullet(\VV;L_1,L_2))^{-1}(U)\simeq \Ber(L_1)^{-1}\ot \Ber(L_2)^{-1}|_U$$
to be the canonical (even) nonvanishing section corresponding to the exact complex $C^\bullet(\VV;L_1,L_2)|_U$.
Equivalently, $\th(L_1,L_2)$ is the Berezinian of the isomorphism of bundles
$$L_1|_U\to L_2^\vee|_U$$
induced by the symplectic form. Note that the line bundle $\Ber(C^\bullet(\VV;L_1,L_2))$ has rank $1|0$, so
its even sections are given locally by even functions.

Similarly to the even case, a quasi-isomorphism of exact complexes leads to an isomorphism of their Berezinians compatible
with their canonical sections coming from the Berezinian of the differential.
Thus, in the context of Lemma \ref{reduction-lem},
the sections $\th(L_1,L_2)$ and $\th(\ov{L}_1,\ov{L}_2)$ get identified under the isomorphism between the corresponding
line bundles.

\begin{theorem}\label{Lagr-int-Ber-thm} 
Let $L_1,L_2\sub \VV$ be a pair of maximal isotropic subbundles in a symplectic bundle over $S$, 
transversal over some open $j:U\hra S$ such that
the natural map $\OO_S\to j_*\OO_U$ is injective.
Assume that for some point $s\in S$,
one has 
$$L_{1s}^+\cap L_{2s}^+=0.$$
Then locally near $s$ there exists a trivialization of $\Ber(C^\bullet(\VV;L_1,L_2))$ such that $\th^{-1}(L_1,L_2)$ corresponds to $f^2$, where $f$
is a regular function. Furthermore, $f$ belongs to the ideal $\II(C^\bullet(\VV,L_1,L_2))$. In other words,
$f\de^{-1}$ is regular near $s$, where $\de:L_1\to \VV/L_2$ is the natural map.
\end{theorem}

\begin{proof} Locally near $s$ we can define a decomposition $L_1=L_1^+\oplus L_1^-$ where $L_1^+$ has rank $m|0$
and $L_2^-$ has rank $0|n$. 
Let us consider the isotropic subbundle $L_1^+\sub \VV$.
Then our assumption implies that in a neighborhood of $s$ the morphism $L_2\to (L_1^+)^\vee$ is surjective,
so by Lemma \ref{reduction-lem}, we have a quasi-isomorphism of $C^\bullet(\VV;L_1,L_2)$ with
$C^\bullet(\ov{\VV}, \ov{L}_1,\ov{L}_2)$, where $\ov{\VV}=(L_1^+)^\perp/L_1^+$ and
$\ov{L}_1=L_1/L_1^+$ has rank $0|n$. 
Now by Proposition \ref{selfdual-rep-prop}, there exists a quasi-isomorphism
$$C^\bullet(\ov{\VV},\ov{L}_1,\ov{L}_2)\simeq [M\rTo{\phi} M^\vee],$$
with $M$ of rank $0|n'$ and $\phi$ selfdual. But this means that actually
$\th^{-1}(L_1,L_2)=\ber(\phi)^{-1}$ is given by the determinant of a skew-symmetric matrix with entries in $\OO_S^+$. 
It remains to set 
$$f=\Pfaff(\phi),$$
where $\Pfaff(?)$ denotes the Pfaffian. 

For the last assertion, we recall that by Lemma \ref{complex-support-lem}(i), one has
$$\II(C(\VV,L_1,L_2))\simeq \II(M\rTo{\phi}M^\vee).$$
Now the assertion follows from Lemma \ref{complex-support-lem}(iv) together with the well known fact
that $\Pfaff(\phi)\phi^{-1}$ is regular (see e.g., \cite[Eq.\ (10)]{Kriv}).
\end{proof}

We need a simple observation that allows to study $\th(L_1,L_2)$ in the presence of a Lagrangian splitting
$V=\La\oplus \La'$.

\begin{lemma}\label{triple-Lag-lemma}
Let $V=\La\oplus \La'$ be a splitting of a symplectic bundle into Lagrangian subbundles, and let
$L_1,L_2\sub V$ be Lagrangian subbundles. Assume that the pairs $(L_1,L_2)$, $(L_1,\La)$ and $(L_2,\La)$
are transversal, and let 
$$\tau_i:\La'\to \La, \ i=1,2,$$
be symmetric (with respect to the duality $\La'\simeq \La^\vee$) morphisms such that
$L_i$ is the graph of $\tau_i$.
Then under the identification $\Ber(\La')\simeq\Ber(\La)^{-1}$,
we have equality of sections of $\Ber(L_1)^{-1}\ot\Ber(L_2)^{-1}$,
$$\th(L_1,L_2)=\pm\th(L_1,\La)\th(L_2,\La)\ber(\tau_1-\tau_2).$$
\end{lemma}

\begin{proof}
By the definition, $\pm\th(L_1,L_2)\th(L_1,\La)^{-1}\th(L_2,\La)^{-1}$ is the Berezinian
of the pairing on $\La'$ given as the composition
$$\La'\ot \La'\rTo{\a_1\ot \a_2} L_1\ot L_2\rTo{(\cdot,\cdot)} \OO,$$
where $(\cdot,\cdot)$ is the symplectic form on $V$, and $\a_i:\La'\to L_i$ are isomorphisms given by
$$\a_i:\La'\to L_i: \la'\mapsto \tau_i(\la')+\la'.$$
Using the fact that $\tau_2$ is symmetric, we compute
\begin{align*}
&(\a_1(\la'_1),\a_2(\la'_2))=(\tau_1(\la'_1)+\la'_1,\tau_2(\la'_2)+\la'_2)=(\tau_1(\la'_1),\la'_2)-(\tau_2(\la'_2),\la'_1)\\
&=((\tau_1-\tau_2)(\la'_1),\la'_2).
\end{align*}
This implies our formula.
\end{proof}

\subsection{Classical picture for theta-characteristic}\label{theta-char-sec}

We are inspired by the following classical isotropic intersection picture for theta-characteristics (see e.g., 
\cite{Mumford1972}, \cite{Harris1982}).

Let $C$ be a curve with a theta-characteristic $\LL$, and let $p\in C$ be a point, $D$ a formal punctured disk around $p$.
Then we have a natural symmetric pairing on $H^0(D,\LL)$ given by
$$b(s,t)=\Res_p(st),$$
where we use the isomorphism $\LL^{\ot 2}\simeq \om_C$ to view $st$ is a $1$-form on $D$.
This induces a nondegenerate pairing on 
$$V=H^0(C,\LL(np)/\LL(-np))$$ 
for any $n>0$, such that the subspace
$$L_2:=H^0(C,\LL/\LL(-np))\sub V$$
is maximal isotropic.
Furthermore, for $n$ large enough, the map $H^0(C,\LL(np))\to V$ is injective, and one can check that
$$L_1=H^0(C,\LL(np))\sub V$$
is also maximal isotropic.
Now the complex $C^\bullet(V;L_1,L_2)=[L_1\to V/L_2]$ can be identified with the complex
$$H^0(C,\LL(np))\to H^0(C,\LL(np)/\LL)$$
computing the cohomology $H^*(C,\LL)$.

This also makes sense for families $(C,\LL)/S$. If we further assume that generically $\LL$ has no nonzero sections
then Proposition \ref{selfdual-rep-prop} implies that locally we can find a quasi-isomorphism of $R\pi_*\LL$
with a complex of the form
$$M\rTo{\phi} M^\vee$$
on $S$, with $\phi^*=-\phi$. It follows that we can use the Pfaffian of $\phi$ as a local equation of the locus where
$\LL$ acquires nonzero sections. 

Applying this to a universal family over the moduli space $\SS_{\bos}^+$ of curves with even theta-characteristics,
we get a canonical structure of an effective Cartier divisor on the locus of $(C,\LL)$ with $h^0(\LL)>0$.
Furthermore, the first order computation in \cite[Sec.\ 1]{Nagaraj1990} implies that this divisor is smooth at all points $(C,\LL)$
where $h^0(\LL)=2$.

\subsection{Symplectic structure for the usual family of curves}
\label{even-symp-sec}

For clarity let us consider the even picture first. 

Note that for a family of curves $\pi:C\to S$ the Hodge bundle $\pi_*\om_C$ embeds as a Lagrangian subbundle
into the symplectic bundle $R^1\pi_*\C_{C/S}$, where $\C_{C/S}=\pi^{-1}\OO_S$. 
Assuming that we have a marked point $p:S\to C$, we will construct for every $n\gg 0$, a bigger symplectic bundle $\VV=\VV_n$,
defined only in terms of data in the formal neighborhood of $p$. 
It will be equipped with an isotropic subbundle $M\sub \VV$ and a Lagrangian subbundle $L\sub\VV$ such that
the reduced symplectic bundle $M^\perp/M$ is identified with $R^1\pi_*\C_{C/S}$ and $L\cap M^{\perp}$ gets identified
with $\pi_*\om_C$. 

We have a natural skew-symmetric form $B_0$ on 
$H^0(D,\OO)$, where $D$ is a formal punctured disk around $p$, given by
$$B_0(f,g)=\Res_p(fdg).$$
We identify $p$ with the subscheme $p(S)\sub C$; in particular, we denote by $np$ the corresponding divisors in $C$ supported on on $p(S)$.

\begin{lemma}\label{sympl-even-lem}
For every $n>0$, such that $R^1\pi_*(\OO(np))=0$, the form $B_0$ induces a symplectic form on the bundle
$$\VV=\VV_n:=\pi_*(\OO_C(np)/(\C_{C/S}+\OO_C(-(n+1)p)))\simeq \coker(\OO_S\to \pi_*(\OO_C(np)/\OO_C(-(n+1)p)))$$ 
of rank $2n$.
Furthermore, $L_{\can}:=\pi_*(\OO_C/\OO(-(n+1)p))/\OO_S$ is a Lagrangian subbundle in $\VV$.
\end{lemma}

\begin{proof}
Let us denote by $\om_{C/S}(np)^{ex}\sub \om_{C/S}(np)$ 
the subsheaf of locally exact forms. Then we have an exact sequence 
$$0\to \C_{C/S}\to \OO_C(np)\rTo{d} \om_{C/S}((n+1)p)^{ex}\to 0$$
which induces an isomorphism of sheaves
$$\OO_C(np)/(\C_{C/S}+\OO_C(-(n+1)p))\rTo{\sim} \om_{C/S}((n+1)p)^{ex}/\om_{C/S}(-np).$$
On the other hand, arguing locally near $p$ we see
that the differential $d$ induces 
a surjective map
$$\pi_*(\OO_C(np)/\OO_C(-(n+1)p)))\to \pi_*(\om_{C/S}((n+1)p)^{ex}/\om_{C/S}(-np))$$
with the kernel $\OO_S$. This shows that the two definitions of $\VV$ agree. The fact that $B_0(f,g)$ is symplectic and $L_{\can}$ is Lagrangian
is clear from the second definition of $\VV$. 
\end{proof}

\begin{lemma} Assume that $n\ge g$ is such that $R^1\pi_*(\OO(np))=0$.
Then the natural map 
$$\wt{L}_{\can}:=\pi_*(\OO_C(np)/\C_{C/S})\to \VV$$ 
is an embedding of a subbundle of rank $n+g$.
\end{lemma}

\begin{proof}
We have an exact sequence
\begin{equation}\label{symp-red-ti-L1-ex-seq}
0\to \OO_S\to \pi_*(\OO_C(np))\to \wt{L}_{\can}\rTo{\ga} R^1\pi_*\C_{C/S}\to 0
\end{equation}
which shows that $\wt{L}_{\can}$ is a bundle of rank $n+g$ (we use Riemann-Roch to see that
$\pi_*(\OO_C(np))/\OO_S$ has rank $n-g$). 
Now the map $\wt{L}_{\can}\to \VV$ can be identified with 
$$\pi_*(\om_{C/S}((n+1)p)^{ex})\to \pi_*(\om_{C/S}((n+1)p)^{ex}/\om_{C/S}(-np)).$$
The fact that it is injective with locally free quotient can be deduced from the long exact
sequence associated with the exact sequence
$$0\to \om_{C/S}(-np)\to \om_{C/S}((n+1)p)^{ex}\to \om_{C/S}((n+1)p)^{ex}/\om_{C/S}(-np)\to 0$$
Indeed, it is enough to see that $\pi_*(\om_{C/S}(-np))=0$ and the sheaves
$R^1\pi_*(\om_{C/S}(-np))$ and $R^1\pi_*(\om_{C/S}((n+1)p)^{ex})$ are locally free.
Since $R^1\pi_*(\OO_C(np))=0$, by Grothendieck-Serre duality, we obtain that $\pi_*(\om_{C/S}(-np))=0$
and $R^1\pi_*(\om_{C/S}(-np))$ is locally free. On the other hand, the exact sequence
$$0\to \om_{C/S}((n+1)p)^{ex}\to \om_{C/S}((n+1)p)\to \OO_p\to 0$$
gives a long exact sequence
$$\pi_*(\om_{C/S}((n+1)p))\to \OO_S\to R^1\pi_*(\om_{C/S}((n+1)p)^{ex})\to R^1\pi_*(\om_{C/S}((n+1)p))=0$$
where the first map is zero by the residue theorem, so $R^1\pi_*(\om_{C/S}((n+1)p)^{ex})\simeq \OO_S$.
\end{proof}

Note that the bundle $R^1\pi_*\C_{C/S}$ has a symplectic structure induced by the multiplication
$$R^1\pi_*\C_{C/S}\otimes R^1\pi_*\C_{C/S}\to R^2\pi_*\C_{C/S}\simeq \OO_S.$$
We are going to show that the symplectic bundle $R^1\pi_*\C_{C/S}$ is obtained from $\VV$ by
a reduction with respect to an isotropic subbundle.
We will deduce this from a more general result concerning the skew-symmetric form on 
$$\WW:=\pi_*(\OO_C(\infty p)/\C_{C/S})$$
induced by $B_0(f,g)=\Res_p(fdg)$.

\begin{prop}\label{symp-H1-prop} 
(i) The map 
$$\ga:\WW\to R^1\pi_*\C_{C/S}$$
is compatible with the skew-symmetric forms.

\noindent
(ii) Assume that $n>0$ is such that $R^1\pi_*(\OO(np))=0$.
Then the subbundle $M:=\pi_*(\OO_C(np))/\OO_S$ in $\VV=\VV_n$ is isotropic and one has
$\wt{L}_{\can}=M^\perp$, so that the restriction of $\ga$ to $\wt{L}_{\can}\sub \WW$ induces an isomorphism
$$M^\perp/M\rTo{\sim}R^1\pi_*\C_{C/S}.$$
\end{prop}

The main idea is to use the following resolution of the relative constant sheaf, $\C_{C/S}$,
$$\C_{C/S}\to [\OO_C(\infty p)\rTo{d} \om_{C/S}(\infty p)^{ex}].$$
We will use also the following second version of the above resolution, due to the quasi-isomorphism
\begin{diagram}
\OO_C(\infty p)&\rTo{d}& \om_{C/S}(\infty p)^{ex}\\
\dTo{}&&\dTo{}\\
\OO_C(\infty p)&\rTo{d}& \om_{C/S}(\infty p)&\rTo{\Res_p} \OO_p
\end{diagram}

The map $\ga$ is induced by the morphism in the derived category
$$\wt{\ga}:\OO_C(\infty p)/\C_{C/S}\to \C_{C/S}[1],$$
represented by the map of complexes
\begin{diagram}
&&\OO_C(\infty p)/\C_{C/S}\\
&&\dTo{d}\\
\OO_C(\infty p)&\rTo{d}& \om_{C/S}(\infty p)^{ex}
\end{diagram}
Without loss of generality we can assume $S$ to be affine. 
Given a section
$\ov{f}\in H^0(S,\WW)$ $=H^0(C,\OO_C(\infty p)/\C_{C/S})$, we want to describe the corresponding
morphism in the derived category 
\begin{equation}\label{ga-ov-f-mult-eq}
\C_{C/S}\rTo{\cdot\ga(\ov{f})} \C_{C/S}[1]
\end{equation}
given by the multiplication with $\ga(\ov{f})$.

We need the following routine technical result.

\begin{lemma}\label{Cech-gen-lem} Let $\AA=[\AA^0\rTo{d}\AA^1]$ be a sheaf of dg-algebras on a topological space $Y$, and
let $c$ be a \v Cech $1$-cocycle with coefficients in $\AA$ (viewed as a complex of sheaves), with respect
to some open covering $(U_i)$ over $Y$: 
$$c=((a^0_{ij}),(a^1_i)),$$
with $a^0_{ij}\in \AA^0(U_i\cap U_j)$, $a^1_i\in\AA^1(U_i)$, so that $d(a^0_{ij})=a^1_j-a^1_i$, $a^0_{ij}+a^0_{jk}=a^0_{ik}$. 
We can think of the cohomology class $[c]$ as a morphism $\Z_Y\to \AA[1]$
in derived category of sheaves, so we have the induced map 
$$r_{[c]}: \AA[-1]\to \AA$$
given by the right multiplication with $c$.
Let $C(\AA)$ be the sheaf \v Cech complex associated with $\AA$, equipped with the standard quasi-isomorphism
$$\iota:\AA\to C(\AA).$$ 
Then $r_{[c]}$ is represented by the chain map 
$$r_c: \AA[-1]\to C(\AA)$$
with the components 
$$\AA^0\rTo{\cdot c} C^1(\AA^0)\oplus C^0(\AA^1)$$
$$\AA^1\rTo{(0,\cdot (-a^0_{ij}))} C^2(\AA^0)\oplus C^1(\AA^1).$$
\end{lemma}

\begin{proof}
The map $r_{[c]}$ is the composition
$$\AA[-1]\simeq \AA\ot\Z_Y[-1]\rTo{\id\ot [c]} \AA\ot\AA\rTo{\mu} \AA,$$
where $\mu$ is the multiplication map.
Since $[c]$ is represented by the chain map $c:\Z_Y[-1]\to C(\AA)$,
it is enough to construct a chain map $\mu':\AA\ot C(\AA)\to C(\AA)$, such that
we have a commutative diagram
\begin{diagram}
\AA\ot \AA&\rTo{\mu}&\AA\\
\dTo{\id\ot\iota}&&\dTo{\iota}\\
\AA\ot C(\AA)&\rTo{\mu'}&C(\AA)
\end{diagram}
Then as $r_c$ we would take the composition of chain maps
$$\AA[-1]\simeq \AA\ot\Z_Y[-1]\rTo{\id\ot c} \AA\ot C(\AA)\rTo{\mu'} C(\AA).$$
The map $\mu'$ is constructed in a straightforward way:
$$\mu'(a\ot (a_{i_1\ldots i_k}))=(a|_{U_{i_1\ldots i_k}}\cdot a_{i_1\ldots i_k}).$$
This gives the claimed map for $r_c$.
\end{proof}

\begin{lemma}\label{H1-mult-lem}
The morphism \eqref{ga-ov-f-mult-eq} is represented by the chain map of resolutions
\begin{diagram}
&& \OO_C(\infty p)&\rTo{d}& \om_{C/S}(\infty p)^{ex}\\
&&\dTo{\cdot d\ov{f}}&&\dTo{-\Res_p(?\cdot f)}\\
\OO_C(\infty p)&\rTo{d}&\om_{C/S}(\infty p)&\rTo{\Res_p}&\OO_p
\end{diagram}
\end{lemma}

\begin{proof}
Note that since the natural map of complexes
$$[\OO_C\to \om_{C/S}]\to [\OO_C(\infty p)\to\om_{C/S}(\infty p)^{ex}]$$
is a quasi-isomorphism, it is enough to prove a similar result using the resolution
$[\OO_C\to \om_{C/S}]$ in the source of the map. 

Now let us consider the resolution for $\C_{C/S}$ obtained by first replacing
it with $[\OO\to \om_{C/S}]$ and then taking the sheaf \v Cech resolution corresponding
to the covering $C=U_0\cup U_1$, where $U_0$ is a neighborthood of $p$ and
$U_1=C\setminus p$. This resolution looks like
$$C^0(\OO)\to C^1(\OO)\oplus C^0(\om_{C/S})\to C^1(\om_{C/S}).$$
Given a closed cycle $c=(f_{01};\a_0,\a_1)\in C^1(\OO)\oplus C^0(\om_{C/S})$,
so that $\a_1-\a_0=df_{01}$ on $U_0\cap U_1$, the multiplication
by $[c]$ map $\C_{C/S}\to \C_{C/S}[1]$ is represented by the chain map
\begin{equation}\label{c-Cech-mult-map}
\begin{diagram}
&&\OO&\rTo{d}&\om_{C/S}\\
&&\dTo{\cdot c}&&\dTo{\cdot (-f_{01)}}\\
C^0(\OO)&\rTo{}&C^1(\OO)\oplus C^0(\om_{C/S})&\rTo{}&C^1(\om_{C/S})
\end{diagram}
\end{equation}
(see Lemma \ref{Cech-gen-lem}).
Now, we observe that there is a natural quasi-isomorphism 
\begin{equation}\label{Cech-quasi-isom-eq}
\begin{diagram}
C^0(\OO)&\rTo{}&C^1(\OO)\oplus C^0(\om_{C/S})&\rTo{}&C^1(\om_{C/S})\\
\dTo{}&&\dTo{}&&\dTo{\Res_p}\\
\OO_C(\infty p)&\rTo{d}&\om_{C/S}(\infty p)&\rTo{\Res_p}&\OO_p
\end{diagram}
\end{equation}
induced by the natural projections $C^0(\OO)\to \OO_C(\infty p)$,
$C^0(\om_{C/S})\to \om_{C/S}(\infty p)$ and by the residue map
$C^1(\om_{C/S})\to \OO_p$.
Composing it with the map \eqref{c-Cech-mult-map} we get that 
the multiplication is represented by the map
\begin{diagram}
&&\OO&\rTo{d}&\om_{C/S}\\
&&\dTo{\cdot \a_1}&&\dTo{-\Res_p(f_{01}\cdot ?)}\\
\OO_C(\infty p)&\rTo{d}&\om_{C/S}(\infty p)&\rTo{\Res_p}&\OO_p
\end{diagram}

Next, we claim that the class $\ga(\ov{f})$ is represented by the
closed cycle 
$$c_{\ov{f}}=(f_{01};0,d\ov{f})\in C^1(\OO)\oplus C^0(\om_{C/S}),$$
where $f_{01}\in \OO(U_0\cap U_1)$ is such that $f_{01}\equiv \ov{f}\mod \C_{C/S}$.
Indeed, this immediately follows from the fact that the image of $c_{\ov{f}}$ under
the quasi-isomorphism \eqref{Cech-quasi-isom-eq} is the cycle $d\ov{f}$ in the middle
term. Applying the above calculation to $c_{\ov{f}}$ we get the result.
\end{proof}

\begin{proof}[Proof of Proposition \ref{symp-H1-prop}]
(i) Given $\ov{f}\in H^0(C,\OO_C(\infty p)/\C_{C/S})$, the pairing with $\ov{f}$ on $H^0(S,\WW)$
is induced by the morphism of sheaves
$$R(\ov{f}):\OO_C(\infty p)/\C_{C/S}\rTo{\Res_p(?\cdot d\ov{f})}\OO_p.$$
The resolution 
$$\C_{C/S}\to [\OO_C(\infty p)\rTo{d} \om_{C/S}(\infty p)\rTo{\Res_p}\OO_p]$$
gives a morphism $\de:\OO_p\to \C_{C/S}[2]$ which induces an isomorphism
$$\OO_S=\pi_*\OO_p\rTo{\sim} R^2\pi_*(\C_{C/S}).$$
Thus, it is enough to check the commutativity of the following square in the derived category:
\begin{diagram}
\OO_C(\infty p)/\C_{C/S}&\rTo{R(\ov{f})}&\OO_p\\
\dTo{\wt{\ga}}&&\dTo{\de}\\
\C_{C/S}[1]&\rTo{\cdot \ga(\ov{f})}& \C_{C/S}[2]
\end{diagram}
By Lemma \ref{H1-mult-lem},
the composition of $\cdot \ga(\ov{f})$ with $\wt{\ga}$ corresponds to the following
composition of maps of complexes
\begin{diagram}
&&&&\OO_C(\infty p)/\C_{C/S}\\
&&&&\dTo{d}\\
&&\OO_C(\infty p)&\rTo{d}& \om_{C/S}(\infty p)^{ex}\\
&&\dTo{\cdot d\ov{f}}&&\dTo{-\Res_p(?\cdot f)}\\
\OO_C(\infty p)&\rTo{d}&\om_{C/S}(\infty p)&\rTo{\Res_p}&\OO_p
\end{diagram}
Hence, this composition is equal to the composition of $\de$ with the map
$$\OO_C(\infty p)/\C_{C/S}\to \OO_p: \ov{g}\mapsto -\Res_p(\ov{f}d\ov{g}).$$
But this is equal to $\ov{g}\mapsto \Res_p(\ov{g}d\ov{f})$ which is exactly $R(\ov{f})$.

\noindent
(ii) The exact sequence \eqref{symp-red-ti-L1-ex-seq}
shows that $M\sub \wt{L}_{\can}$ is the kernel of the surjective map
$$\ga:\wt{L}_{\can}\to R^1\pi_*\C_{C/S}.$$
Hence, by (i), we have an inclusion $\wt{L}_{\can}\sub M^\perp$.
Since these are subbundles of the same rank in $\VV$, it follows that $\wt{L}_{\can}=M^\perp$.
\end{proof}

For a Lagrangian subbundle $\Lambda\sub R^1\pi_*\C_{C/S}$, let us set
\begin{equation}\label{L1-def-eq}
L_\Lambda:=\ga^{-1}(\Lambda)\sub \wt{L}_{\can}\sub \VV.
\end{equation}

\begin{prop} \label{Lagr-quasi-isom-prop}
Assume that $n>0$ is such that $R^1\pi_*(\OO(np))=0$.
Given a Lagrangian subbundle $\Lambda\sub R^1\pi_*\C_{C/S}$, let us define $L_\Lambda$ by \eqref{L1-def-eq}
and $L_{\can}$ as in Lemma \ref{sympl-even-lem}.  
Then $L_\Lambda$ and $L_{\can}$ are Lagrangian subbundles in $\VV$, and
the complex $C^\bullet(\VV;L_\Lambda,L_{\can})$ is quasi-isomorphic to
$$[\Lambda\to R^1\pi_*\OO_C],$$
where the morphism is the composition $\Lambda\hra R^1\pi_*\C_{C/S}\to R^1\pi_*\OO_C$.
\end{prop}

\begin{proof} The first assertion follows from Proposition \ref{symp-H1-prop}(ii) and Lemma \ref{sympl-even-lem}.
To deduce the second, we use Lemma \ref{reduction-lem}.
We need to check that the natural map 
$$\a:M^\perp=\pi_*(\OO_C(np)/\C_{C/S})\to \VV/L_{\can}\simeq \pi_*(\OO_C(np)/\OO_C)$$
is surjective. This map fits into the following commutative diagram with exact rows:
\begin{diagram}
\pi_*(\OO_C(np))&\rTo{}&\pi_*(\OO_C(np)/\OO_C)&\rTo{}&R^1\pi_*(\OO_C)\\
\dTo{}&&\dTo{\id}&&\dTo{\ga}\\
\pi_*(\OO_C(np)/\C_{C/S})&\rTo{\a}&\pi_*(\OO_C(np)/\OO_C)&\rTo{\b}&R^1\pi_*(\OO_C/\C_{C/S})
\end{diagram}
Now we observe that the map $\ga$ fits into an exact sequence
$$R^1\pi_*(\C_{C/S})\to R^1\pi_*(\OO_C)\rTo{\ga} R^1\pi_*(\OO_C/\C_{C/S})$$
which implies that $\ga=0$ since the previous map in the exact sequence is surjective.
Hence, in the above commutative diagram we should have $\b=0$, and so $\a$ is surjective, as claimed.

Furthermore, we have 
$$L_{\can}\cap M^\perp=L_{\can}\cap \wt{L}_{\can}=\pi_*(\OO_C/\C_{C/S})\simeq \pi_*(\om_{C/S}).$$
Thus, from Lemma \ref{reduction-lem} we get a quasi-isomorphism
$$C^\bullet(\VV;L_\Lambda,L_{\can})\simeq C^\bullet(R^1\pi_*\C_{C/S};\Lambda,\pi_*(\om_{C/S})).$$
Finally, the exact sequence
\begin{equation}\label{R1pi*-ex-seq}
0\to \pi_*(\om_{C/S})\to R^1\pi_*\C_{C/S}\to R^1\pi_*\OO_C\to 0
\end{equation}
gives an isomorphism
$C^\bullet(R^1\pi_*\C_{C/S};\Lambda,\pi_*(\om_{C/S}))\to [\Lambda\to R^1\pi_*\OO_C]$.
\end{proof}

Note that the bundle $R^1\pi_*\C_{C/S}$ comes from a local system over $S$
equipped with a real structure.
If we choose $\Lambda$ coming from a real Lagrangian local subsystem then the map
$\Lambda\to R^1\pi_*\OO_C$ will be an isomorphism, hence in this case the complex
$C^\bullet(\VV;L_\Lambda,L_{\can})$ is acyclic.

\subsection{Symplectic structures in supercase}
\label{sympl-super-sec}

Assume now that $\pi:X\to S$ is a family of supercurves, and
we have an effective Cartier divisor $p\sub X$ of codimension $1|0$
with the support $\ov{p}$ (as observed in \cite{DRS}, the choice of $p$ is equivalent to a choice of section $S\to X$
supported on $\ov{p}$, see also \cite[Sec.\ 9.3]{Witten}). Then we can define sheaves $\OO_X(mp)$ for $m\in \Z$.
Now we would like to mimic the
definition of a symplectic bundle $\VV=\VV_n$ from Sec.\ \ref{even-symp-sec}, where $n>0$ is such that
$$R^1\pi_*(\OO_X(np))=\pi_*(\OO_X(-np))=0.$$
We can start by taking the vector bundle 
$$\wt{\VV}=\pi_*(\OO_X(np)/\OO_X(-Np))$$
for some $N\gg n$, and observe that the form
$$B(f,g)=\Res_p(f\de g)$$
descends to a skew-symmetric form on $\wt{\VV}$
(recall that $\de:\OO_X\to\om_{X/S}$ is the canonical derivation).
Now we denote by $\VV$ the quotient of $\wt{\VV}$ by the kernel of this form.
Locally over $S$ we can choose an isomorphism of a formal neighborhood of $p$ in $X$ with the split
formal superdisk $D$ over $S$ with the coordinate $(z,\th)$ (see Remark \ref{superdisk-remark}).
Then 
$$\wt{\VV}|_D=z^{-n}\OO_S[\![z]\!]/z^N\OO_S[\![z]\!] \oplus z^{-n}\OO_S[\![z]\!]\th/z^N\OO_S[\![z]\!]\th,$$
and the kernel of the form gets identified with 
$$z^{n+1}\OO_S[\![z]\!]/z^N\OO_S[\![z]\!]\oplus z^n\OO_S[\![z]\!]\th/z^N\OO_S[\![z]\!]\th,$$
so that
$$\VV|_D=\coker(\OO_S\to z^{-n}\OO_S[\![z]\!]/z^{n+1}\OO_S[\![z]\!] \oplus z^{-n}\OO_S[\![z]\!]\th/z^n\OO_S[\![z]\!]\th).$$
This shows that $\VV$ is a symplectic bundle of rank $2n|2n$.
 
Furthermore, we can still define the Lagrangian subbundle $L_{\can}\sub \VV$ by 
$$L_{\can}:=\pi_*(\OO_X/\OO_X(-(n+1)p))/\OO_S,$$ 
and the Lagrangian subbundle $L_\Lambda\sub \VV$
as the preimage of a Lagrangian subbundle $\Lambda\sub R^1\pi_*\C_{X/S}$
under the map 
$$\wt{L}_{\can}:=\pi_*(\OO_X(np)/\C_{X/S})\to R^1\pi_*\C_{X/S}.$$
We also have an isotropic subbundle
$$M:=\pi_*(\OO_X(np))/\OO_S\sub \wt{L}_{\can}.$$

Most of the arguments of Section \ref{even-symp-sec} go through in the supercase.
However, the assertion of Proposition \ref{Lagr-quasi-isom-prop} only holds over an open subset $U\sub S$
where $\OO_S\to \pi_*\OO_X$ is an isomorphism since we want to
use the exact sequence \eqref{R1pi*-ex-seq}. 

To describe the situation globally, let us consider the complex (concentrated in degrees $0$ and $1$),
$$\wt{\Lambda}=[\pi_*\OO_X(np)\to L_\Lambda].$$
It has locally free cohomology, so locally 
$$\wt{\Lambda}\simeq \OO_S\oplus \Lambda[-1].$$
Then we have a natural morphism in derived category
$$\wt{\Lambda}\to R\pi_*\OO_X$$
represented by the natural map of complexes
$$[\pi_*\OO_X(np)\to L_\Lambda]\to [\pi_*\OO_X(np)\to \pi_*(\OO_X(np)/\OO_X)].$$
Now we have the following replacement of Proposition \ref{Lagr-quasi-isom-prop}.

\begin{prop}\label{super-Lagr-quasi-isom-prop} For large enough $n$, and a 
Lagrangian subbundle $\Lambda\sub R^1\pi_*\C_{X/S}$, the subbundles
$L_\Lambda$ and $L_{\can}$ in $\VV$, defined above are Lagrangian.
Furthermore, we have $M^\perp=\wt{L}_{\can}$ and the natural map
$$M^\perp/M\to R^1\pi_*\C_{X/S}$$
is an isomorphism and identifies $L_\Lambda/M$ with $\Lambda\sub R^1\pi_*\C_{X/S}$.
Over an open subset $U\sub S$ where the theta-characteristic associated with $X_s$ has no global sections,
we have $L_{\can}+M^\perp=\VV$ and the reduced Lagrangian
$L_{\can}\cap M^\perp$ gets identified with $\pi_*\om_{X/S}\sub R^1\pi_*\C_{X/S}$.
There is an isomorphism in the derived category of $S$,
$$\Cone(\wt{\Lambda}\to R\pi_*\OO_X)\simeq C^\bullet(\VV;L_\Lambda,L_{\can}).$$ 
\end{prop}

\begin{proof}
Most of the assertions are checked as in the even case.
To prove the last assertion, we note that the complex $C^\bullet(\VV;L_\Lambda,L_{\can})$ has form 
$$[L_\Lambda\to \VV/L_{\can}]\simeq [L_\Lambda\to \pi_*(\OO_X(np)/\OO_X)].$$
Now we have a natural surjective morphism of complexes
$$\Cone(\wt{\Lambda}\to [\pi_*\OO_X(np)\to \pi_*\OO_X(np)])\to [L_\Lambda\to \pi_*(\OO_X(np)/\OO_X)]$$
with the kernel $[\pi_*\OO_X(np)\rTo{\id}\pi_*\OO_X(np)][1]$. This immediately implies our claim.
\end{proof}

\subsection{Application to the moduli space of supercurves and poles of the period matrix}

Let $\SS$ denote the moduli space of supercurves of 
fixed genus $g\ge 2$, and let $\SS_{\bos}$ denote the corresponding reduced space.
Let $\SS^+$ denote the connected component over the component $\SS_{\bos}^+$ corresponding to even spin structures.
There is a natural reduced divisor $\DD_0\sub \SS_{\bos}^+$, corresponding to spin structures with a nonzero section.

Below
we work over some affine superscheme $S$ equipped with an \'etale morphism $S\to \SS^+$. Abusing the notation we will still
denote by $\DD_0$ the preimage of $\DD_0\sub\SS_{\bos}^+$ in $S_{\bos}$.
We have a well defined line bundle $\Ber(R\pi_*\OO_X)$ of rank $1|0$, 
where $\pi:X\to S$ is the pullback of the universal supercurve. Let $U\sub S$ be the open complement to $\DD_0$.
Then over $U$, $R^1\pi_*\OO_X$ is a vector bundle of rank $g|0$ and $\pi_*\OO_X=\OO_S$ (see Proposition \ref{loc-free-prop}).
Thus, we have a canonical isomorphism
\begin{equation}\label{Ber-R1pi*-isom}
\Ber(R\pi_*\OO_X)|_U\simeq \Det(R^1\pi_*\OO_X)^{-1}|_U.
\end{equation}

Now assume that over $S$ we have a Lagrangian subbundle $\Lambda\sub R^1\pi_*\C_{X/S}$
such that the induced map $\Lambda|_U\to R^1\pi_*\OO_X|_U$
is an isomorphism. Then the pair of Lagrangian subbundles $\pi_*\om_{X/S}|_U$ and $\Lambda|_U$ is
transversal, so using the construction of Section \ref{det-sec}, we have the corresponding nonvanishing section
$$\th_\Lambda:=\th(\pi_*\om_{X/S},\Lambda)\in \Gamma(U,\Det(\Lambda)^{-1}\ot \Ber(R\pi_*\OO_X)^{-1}|_U),$$
where we used the isomorphism \eqref{Ber-R1pi*-isom}. Note that the line bundle on the right is defined not just over $U$,
so it makes sense to ask about the behavior of $\th_\Lambda$ near the divisor $\DD_0$.

Assume that $\La'\sub R^1\pi_*\C_{X/S}$ is a  Lagrangian complement to $\Lambda$ defined over $S$, so that we have a decomposition
$$R^1\pi_*\C_{X/S}=\La\oplus\La'.$$
Then over $U$ the subbundle $\pi_*(\om_{X/S})\sub R^1\pi_*\C_{X/S}$ is the graph of
a morphism 
$$\Om:\Lambda'\to\Lambda.$$ 

\begin{theorem}\label{square-root-thm} Assume that $\La$ is transversal to 
$$\pi_*\om_{C/S_{\bos}}\sub R^1\pi_*\C_{C/S_{\bos}}=R^1\pi_*\C_{X/S}.$$
Then locally on $S$ near any point of the divisor $\DD_0$ 
there exists a trivialization of $\Det(\Lambda)^{-1}\ot \Ber(R\pi_*\OO_X)^{-1}$,
such that $\th_\Lambda^{-1}=f^2$ for some regular function $f$ vanishing on $\DD_0$.
Furthermore, $f\Om$ is regular on $S$.
At a generic point of $\DD_0$, $f$ projects to a local equation of $\DD_0$ in $S_{\bos}$.
\end{theorem}

\begin{proof} We will deduce this from Theorem \ref{Lagr-int-Ber-thm}. 
Over $U$ the complex $\Cone(\wt{\Lambda}\to R\pi_*\OO_X)$ is exact, and $\th_\Lambda$ can be identified
with the canonical nonvanishing section
$$\wt{\th}\in\Ber\Cone(\wt{\Lambda}\to R\pi_*\OO_X)^{-1}|_U\simeq \Det(\Lambda)^{-1}\ot \Ber(R\pi_*\OO_X)^{-1}|_U.$$
Now by Proposition \ref{super-Lagr-quasi-isom-prop}, we have a quasi-isomorphism of
$\Cone(\wt{\Lambda}\to R\pi_*\OO_X)$ with the complex
$C^\bullet(\VV;L_\Lambda,L_{\can})$, so that $\wt{\th}$
gets identified with $\th(L_\Lambda,L_{can})$. We claim that
the assumptions of Theorem \ref{Lagr-int-Ber-thm} are satisfied for the pair of Lagrangians
$(L_\Lambda,L_{\can})$. This will immediately imply the first assertion.
It is easy to see that the fibers of our Lagrangians in $\VV$
can be identified with the similarly defined subspaces of the symplectic vector space
$$V:=H^0(\OO_C(np)/\OO_C(-(n+1)p))/\C \oplus H^0(L(np)/L(-np)),$$
where $(C,\OO_C\oplus L)$ is the split supercurve corresponding to $s$.
Namely, the fiber of $L_\Lambda$ is the preimage of $\Lambda_s\sub H^1(C,\C)$ 
under the map 
$$H^0(\OO_C(np)/\C_C)\oplus H^0(L(np))\to H^1(C,\C),$$
so its even part is the preimage of $\Lambda_s$ in $H^0(\OO_C(np)/\C_C)\sub V^+$.
On the other hand, the fiber of $L_{\can}$ is the subspace
$$H^0(\OO_C/\OO_{C(-(n+1)p)}/\C)\oplus H^0(L/L(-np))\sub V,$$
so its even part is $H^0(\OO_C/\OO_{C(-(n+1)p)}/\C)\sub V^+$. 
Thus, the intersection of even parts of the Lagrangians is the intersection of the preimage of $\Lambda_s$
with $H^0(\OO_C/\C_S)$. But the latter space projects to $H^0(\omega_C)\sub H^1(C,\C)$ which intersects
trivially with $\Lambda_s$. 

Let us choose a splitting $\La\to L_\Lambda$ of the projection $L_\Lambda\to \La$. Then we get a composed morphism
$$\La\to L_\Lambda\to \pi_*(\OO_X(np)/\OO_X)$$
and an induced quasi-isomorphism of complexes
$$[\La\oplus \pi_*(\OO_X(np))/\OO_S\rTo{\de} \pi_*(\OO_X(np)/\OO_X)]\simeq \Cone(\wt{\La}\to R\pi_*\OO_X).$$
Hence, combining Proposition  \ref{super-Lagr-quasi-isom-prop} 
with Theorem \ref{Lagr-int-Ber-thm} we obtain that $f\de^{-1}$ is regular on $S$.

Now let $L'_1\sub \wt{L}_{\can}$ be the preimage of $\La'$, so that we have a resolution
$$[\pi_*(\OO_X(np))/\OO_S\to L'_1]\rTo{\sim} \La'.$$
To show that $f\Om$ is regular, it remains to check that $\Om$ is induced by the composed map
$$L'_1\to \pi_*(\OO_X(np)/\C_{X/S})\to \pi_*(\OO_X(np)/\OO_X)\rTo{\de^{-1}} \La\oplus \pi_*(\OO_X(np))/\OO_S\rTo{p_\La} \La$$
which vanishes on $\pi_*(\OO_X(np))/\OO_S\sub L'_1$ (where $p_\La$ is the projection onto $\La$).

To this end, we note that $-\Om$ is the composition
of the maps of bundles over $U$, 
$$\La'\to R^1\pi_*\C_{X/S}\to R^1\pi_*\OO\to\La,$$
where the last arrow is the inverse of the isomorphism $\La\to R^1\pi_*\OO$ over $U$.
Now we observe that the maps
$$\La'\to R^1\pi_*\C_{X/S}\to R^1\pi_*\OO$$
are represented by the maps of resolutions (given by columns) in the following diagram
\begin{diagram}
\pi_*(\OO_X(np))/\OO_S&\rTo{\id}&\pi_*(\OO_X(np))/\OO_S&\rTo{\id}&\pi_*(\OO_X(np))/\OO_S\\
\dTo{}&&\dTo{}&&\dTo{}\\
L'_1&\rTo{}&\pi_*(\OO_X(np)/\C_{X/S})&\rTo{}&\pi_*(\OO_X(np)/\OO_X)
\end{diagram}
Furthermore, over $U$, the inverse to the isomorphism $\La\to R^1\pi_*\OO$ is precisely the
composition of $\de^{-1}$ with the projection to $\La$. Together with the previous observation this proves our claim.

The fact that $f$ projects to a local equation of $\DD_0$ at a generic point follows from the comparison with the construction of
Sec.\ \ref{theta-char-sec}.
\end{proof}

\section{Supermeasure on the moduli space}\label{s-5}

\subsection{Complex conjugation and the pairing}\label{complex-conj-sec}

For a complex manifold $\XX$ we denote by $\ov{\XX}$ conjugate complex manifold.
In other words, $\ov{\XX}=\XX$ as a $C^\infty$-manifold, but the holomorphic functions on $\ov{\XX}$ are antiholomorphic
functions on $\XX$. For every local holomorphic function $f$ on $\XX$ we can view $\ov{f}$ as a holomorphic function on
$\ov{\XX}$. It follows that for every holomorphic vector bundle $\VV$ over $\XX$ we have a holomorphic vector bundle
$\ov{\VV}$ over $\ov{\XX}$ (defined by the complex conjugate transition matrices), so that we have 
a $\C$-antilinear isomorphism
$$H^0(\XX,\VV)\rTo{\sim} H^0(\ov{\XX},\ov{\VV}): s\mapsto \ov{s}.$$
More generally, the correspondence $\XX\mapsto \ov{\XX}$ extends to a functor on the category of complex manifolds,
and for a morphism $f:\XX\to \YY$, we have
$$\ov{f_*\VV}\simeq \ov{f}_*(\ov{\VV}).$$

Given a holomorphic $n$-form $\eta$ on $\XX$, we can view $\ov{\eta}$ as a holomorphic $n$-form on $\ov{\XX}$.
Hence, $p_1^*\eta\we p_2^*\ov{\eta}$ is a holomorphic $2n$-form on $\XX\times\ov{\XX}$,
where $p_1:\XX\times\ov{\XX}\to \XX$ and $p_2:\XX\times\ov{\XX}\to \ov{\XX}$ are the projections.
Now let us consider the diagonal map $\De:\XX\to \XX\times \ov{\XX}$ (which is neither holomorphic nor antiholomorphic).
Then we have
$$\eta\we \ov{\eta}=\De^*(p_1^*\eta\we p_2^*\ov{\eta}).$$

For complex supermanifolds, we cannot define a pointwise conjugation of functions, so we define a complex conjugate formally as follows.

\begin{definition} For a complex supermanifold $\XX=(|\XX|,\OO_\XX)$ of dimension $n|m$ we define the {\it conjugate} complex supermanifold $\ov{\XX}$ of dimension $n|m$ as
$(|\XX|,\ov{\OO_\XX})$, where $\ov{\OO_\XX}$ is the same sheaf $\OO_\XX$ with the complex conjugate $\C$-algebra structure (obtained by composing the standard embedding
$\C\to \OO_\XX$ with the complex conjugation on $\C$).
\end{definition}

In local coordinates $(x_1,\ldots,x_n,$ $\th_1,\ldots,\th_m)$ holomorphic functions on $\ov{\XX}$ form an exterior
algebra in $\th_i$ over the ring of antiholomorphic functions in $x_1,\ldots,x_n$. This easily implies that $\ov{\XX}$ is still a complex supermanifold of dimension $n|m$.

If $\eta$ is a holomorphic section of the holomorphic Berezinian on $\XX$
then $p_1^*\eta \we p_2^*\ov{\eta}$ will be a holomorphic section of the holomorphic Berezinian on $\XX\times\ov{\XX}$. 
(Note that if we locally view $\XX\times\ov{\XX}$ as a real analytic supermanifold then $p_1^*\eta \we p_2^*\ov{\eta}$ 
will be an integral form of codimension $n|0$, where the complex dimension of $\XX$ is $n|m$.)

Now let $V$ be a local system of $\R$-vector spaces of finite rank over $\XX$ with a bilinear pairing $b:V\ot V\to \R$
(where we denote by $\R$ the corresponding constant sheaf).
Assume that
$$\iota:\VV\to V\ot_{\R}\OO_{\XX}$$ 
is a morphism of holomorphic vector bundles on $\XX$.
Then it induces a morphism $\ov{\iota}:\ov{\VV}\hra V\ot_{\R}\OO_{\ov{\XX}}$ of holomorphic vector bundles on $\ov{\XX}$.
Hence, on $\XX\times\ov{\XX}$ we get morphisms
$$p_1^*\iota:p_1^*\VV\hra p_1^{-1}V\ot_{\R}\OO_{\XX\times\ov{\XX}}, \ \ 
p_2^*\ov{\iota}:p_2^*\ov{\VV}\hra p_2^{-1}V\ot_{\R}\OO_{\XX\times\ov{\XX}}.$$

Now we observe that the restrictions of the local systems $p_1^{-1}V$ and $p_2^{-1}V$ to the diagonal in $\XX\times\ov{\XX}$
are naturally isomorphic: both restrictions are identified with $V$.
Therefore, there exists a unique isomorphism
$$p_1^{-1}V\simeq p_2^{-1}V$$
defined in a neighborhood of the diagonal in $\XX\times\ov{\XX}$. Thus, on this neighborhood 
we get an $\OO$-bilinear pairing
$$p_1^*\VV\ot p_2^*\ov{\VV}\to p_1^{-1}V\ot_{\R} p_1^{-1}V\ot_{\R} \OO_{\XX\times\ov{\XX}}\rTo{b\ot\id}\OO_{\XX\times\ov{\XX}}.$$

\subsection{Definition of the supermeasure and regularity theorem}

We can apply the above construction to the embedding $\pi_*\om_{X/\UU}\to R^1\pi_*\C_{X/\UU}$ 
(recall that $\UU\sub \SS^+$ is the locus corresponding to theta-characteristics with trivial $H^0$)
and the identification
\begin{equation}\label{H1-R-trivialization-eq}
R^1\pi_*\C_{X/\SS^+}=R^1\pi_*(\R_X)\ot_{\R}\OO_{\SS^+}.
\end{equation}

A slight modification we can make in this case is that instead of the diagonal in $\SS^+\times\ov{\SS^+}$
it is more natural to consider the locus in the reduced space $\SS_{\bos}^+\times\ov{\SS_{\bos}^+}$ corresponding
to the pairs $((C,L),(C',L'))$ with $C=C'$ (but possibly $L\neq L'$). 
Let us call this locus the {\it quasi-diagonal} in $\SS_{\bos}^+\times\ov{\SS_{\bos}^+}$ and denote it by
$$\De^{qu}(\SS_{\bos}^+)\sub \SS_{\bos}^+\times\ov{\SS_{\bos}^+}.$$
Since the local sistem $R^1\pi_*\R_X$ comes from the usual moduli space of curves, we can apply
the construction of Section \ref{complex-conj-sec} in our case, working on a neighborhood of the quasi-diagonal 
(instead of the diagonal). Thus, we get a pairing between vector bundles,
\begin{equation}\label{pairing-omega-eq}
p_1^*\pi_*\om_{X/\UU}\ot p_2^*\ov{\pi_*\om_{X/\UU}}\to \OO_{\UU\times\ov{\UU}}
\end{equation}
defined near the quasi-diagonal $\De^{qu}(\UU_0)$ in $\UU_0\times\ov{\UU}_0$.

We claim that this pairing is nondegenerate in a neighborhood of the quasi-diagonal.
Indeed, it is enough to check this on points of the quasi-diagonal, where it follows from the nonvanishing
of $\int_C \eta\we\ov{\eta}$ for a nonzero holomorphic $1$-form on a usual complex curve.

From the pairing \eqref{pairing-omega-eq}
we get a morphism of holomorphic bundles
$$p_1^*\pi_*\om_{X/\UU}\to p_2^*\ov{\pi_*\om_{X/\UU}}^\vee,$$
or passing to determinants,
$$p_1^*\Ber_1|_{\UU}\to p_2^*\ov{\Ber_1}^{-1}|_{\UU},$$
where $\Ber_j:=\Ber(R\pi_*\om_{X/\SS}^{\ot j})$
(recall that $R^1\pi_*\om_{X/\UU}\simeq \OO_\UU$, so
$\Ber_1|_{\UU}\simeq \Det(\pi_*\om_{X/\UU})$).
Furthermore, this is an isomorphism in a neighborhood of $\De^{qu}(\UU_0)$.

In other words, considering the inverse of the above isomorphism, we get a canonical holomorphic nonvanishing section 
$$s\in p_1^*\Ber_1\ot p_2^*\ov{\Ber_1}$$
in a neighborhood of $\De^{qu}(\UU_0)$ in $\UU\times \ov{\UU}$.
Note that by definition, $s$ is obtained by applying the construction of Section \ref{det-sec} to the pair of Lagrangians
$(p_1^*\pi_*\om_{X/\UU},p_2^*\ov{\pi_*\om_{X/\UU}})$:
$$s=\th(p_1^*\pi_*\om_{X/\UU},p_2^*\ov{\pi_*\om_{X/\UU}})^{-1}.$$

The Mumford isomorphism in the supercase (see \cite{Voronov})
is an isomorphism from $\Ber_1^5$ to $\Ber_3$, which is the holomorphic Berezinian on $\SS$ (i.e., the Berezinian of the cotangent sheaf).
Using it,
from $s^5$ we get a canonical holomorphic nonvanishing section of the holomorphic Berezinian
on a neighborhood of $\De^{qu}(\UU_0)$ in $\UU\times \ov{\UU}$.
Below we will show that it extends to a meromorphic section on a neighborhood of $\De^{qu}(\SS_{\bos}^+)$
in $\SS^+\times\ov{\SS^+}$.
Moreover, we have the following regularity statement.

\begin{theorem}\label{s5-thm} 
The image of the section $s^5$ under the Mumford isomorphism
extends to a meromorphic section of the holomorphic Berezinian on a neighborhood of the quasi-diagonal in 
$\SS^+\times\ov{\SS^+}$.
If $g\le 11$ then the extended section is regular on a neighborhood of the quasi-diagonal.
\end{theorem}

\subsection{Setup for studying the poles}

To study poles of $s$ near the point corresponding to a supercurve with the underlying curve $C$, we choose a Lagrangian 
splitting of $H^1(C,\R)$. This defines a local Lagrangian splitting of the local system $R^1\pi_*(\R_X)$, and hence
a Lagrangian splitting
\begin{equation}\label{lagrangian-splitting-eq}
R^1\pi_*\C_{X/\SS^+}=W\oplus W'.
\end{equation}
In this case we have the conjugate subbundles
$\ov{W}$, $\ov{W'}$ in $R^1\pi_*\C_{\ov{X}/\ov{\SS^+}}\simeq H^1(C,\R)\ot_{\R}\OO_{\ov{\SS^+}}$.
As before, we work in a neighborhood of the quasi-diagonal in $\SS^+\times\ov{\SS^+}$, where we have the identification
$$p_1^{-1}R^1\pi_*(\R)\simeq p_2^{-1}R^1\pi_*(\R).$$
With respect to this identification, we have
$$p_1^*W=p_2^*\ov{W}, \ \ p_1^*W'=p_2^*\ov{W'}$$
as subbundles in $H^1(C,\R)\ot_{\R}\OO_{\SS^+\times\ov{\SS^+}}$. 

Recall that the choice of $W$ leads to a nonvanishing
section over the complement to the divisor $\DD_0\sub \SS_{\bos}^+$,
$$\th^{-1}_W=\th(\pi_*\om_{X/\UU},W)^{-1}\in \Det(W)\ot \Det(R^1\pi_*\OO)^{-1}\simeq\Det(W)\ot \Ber_1.$$
Applying Lemma \ref{triple-Lag-lemma} to the triple of Lagrangians 
$$(p_1^*\pi_*\om_{X/\UU}, p_2^*\ov{\pi_*\om_{X/\UU}}, p_1^*W=p_2^*\ov{W})$$
we get
\begin{equation}\label{s-main-formula-eq}
\begin{array}{l}
\pm s=\th(p_1^*\pi_*\om_{X/\UU},p_1^*W)^{-1}\th(p_2^*\ov{\pi_*\om_{X/\UU}},p_2^*\ov{W})^{-1}\det(p_1^*\tau-p_2^*\ov{\tau})^{-1}
\\
=p_1^*\th^{-1}_W\cdot p_2^*\ov{\th}^{-1}_W\cdot \det(p_1^*\tau-p_2^*\ov{\tau})^{-1},
\end{array}
\end{equation}
where 
$$\tau:W'\to W\simeq (W')^{\vee}$$ 
is a symmetric morphism such that
$\pi_*\om_{X/\UU}$ is the graph of $\tau$
(so $\tau$ is the period matrix with respect to the above splitting). Formula \eqref{s-main-formula-eq} will allow
us to use information on the behavior of $\th_W$ near the divisor $\DD_0$, to deduce the regularity of $s$ for small genus.


\subsection{Proof of Theorem \ref{s5-thm}}

It is enough to study the situation in a neighborhood of the supercurve corresponding
to a curve with an even theta-characteristic $(C_0,L_0)$ such that $H^0(C_0,L_0)\neq 0$.
In addition to the splitting \eqref{lagrangian-splitting-eq}
into real Lagrangians, we can also define locally a different splitting.
Namely, let $\SS_{\bos}^+$ be the usual moduli stack underlying $\SS^+$ (classifying curves with even theta-characteristics).
Since $R^1\pi_*\C_{X/\SS^+}$ comes from a local system, it is identified with
the pull-back of a similar bundle $R^1\pi_{0*}\C_{C/\SS_{\bos}^+}$ on $\SS_{\bos}^+$, 
where $\pi_0:C\to\SS_{\bos}^+$ is the universal curve over $\SS_{\bos}^+$, 
with respect to some local splitting of the supermanifold $\SS^+$. 

Now, locally over $\SS_{\bos}^+$ we can choose a Lagrangian splitting of the form
$$R^1\pi_{0*}\C_{C/\SS_{\bos}^+}=\La'\oplus \La,$$
with $\La'=\pi_{0*}\om_{C/\SS_{\bos}^+}$.
Let us denote still by $\La$ and $\La'$ the induced Lagrangian subbundles of rank $g|0$ in $R^1\pi_*\C_{X/\SS^+}$.
Furthermore, a convenient choice of the complement $\La$ is
$$\La=W.$$

Over the open subset $\UU\sub\SS^+$, $\pi_*\om_{X/\UU}$ is a Lagrangian subbundle of $R^1\pi_*\C_{X/\UU}$,
whose restriction to $\UU_0=\UU\cap\SS_{\bos}^+$ is $\La'$. Hence, there is a unique symmetric even morphism
$$\Om:\La'\to \La$$
over $\UU$, whose graph is $\pi_*\om_{X/\UU}$. Note that $\Om|_{\UU_0}=0$, so the entries of $\Om$ belong
to $\NN^2$ (where $\NN$ is the nilradical) on $\UU$. 

By Theorem \ref{square-root-thm}, locally 
there exists a function $f$ on $\SS^+$ 
such that $f\Omega$ extends regularly to $\SS^+$ and $f^2$ differs from $\th_{\La}^{-1}=\th_W^{-1}$ by a unit.

Now let us compare our two splittings of $R^1\pi_*\C_{X/\SS^+}$.
There exist a regular even morphism $A:W'\to W$  
over $\SS^+$ (locally) such that
$\La'$ is the graph of $A$. Let us consider the composition
$$\wt{\Om}:W'\rTo{x\mapsto Ax+x}\La'\rTo{\Om} \La=W$$
Note that the entries of $\wt{\Om}$ still belong to $\NN^2$ and $f\wt{\Omega}$ is regular on $\SS^+$.
Now any element of $\pi_*\om_{X/\UU}$ has
form $y+\Om y$ for some $y\in \La'$, so using the decomposition
$y=Ax+x$, where $x\in W'$, we get
$$y+\Om y=x+Ax+\wt{\Om}(x)=x+(A+\wt{\Om})x.$$
This implies that $\pi_*\om_{X/\UU}$ is the graph of $\tau:W'\to W$, with
$$\tau=A+\wt{\Om}.$$

Now let us denote by 
$$\AA(f)\sub j_*\OO_{\UU\times\ov{\UU}}$$ 
the $\OO_{\SS^+\times\ov{\SS^+}}$-subalgebra generated by elements $p_1^*(x/f)$, $p_2^*(\ov{x}/\ov{f})$,
with $x\in\NN^2$.
Note that the entries of $p_1^*\tau$ and $p_2^*\ov{\tau}$ are in $\AA(f)$.
Furthermore, since on $\SS_{\bos}^+$ the Lagrangian subbundles $\pi_{0*}\om_{C/\SS_{\bos}^+}$ and
$\ov{\pi_{0*}\om_{\CC/\SS_{\bos}^+}}$ in $R^1\pi_*(\C_{C/\SS_{\bos}^+})$ are transversal and come from the
usual moduli space of curves, we get 
 that $p_1^*A-p_2^*\ov{A}$ is invertible near the quasi-diagonal. This implies that
$${\det}(p_1^*\tau-p_2^*\ov{\tau})^{-1}\in \AA(f).$$ 

Finally, since $\th_W^{-1}$ differs from $f^2$ by an invertible function, using \eqref{s-main-formula-eq} we obtain
$$s=p_1^*f^2\cdot p_2^*\ov{f}^2\cdot a,$$
with $a\in \AA(f)$. This immediately shows that $s$ is meromorphic.
Furthermore, since 
$$s^5=p_1^*f^{10}\cdot p_2^*\ov{f}^{10}\cdot a^5,$$
using the fact that $a^5\in\AA(F)$, we see that  $s^5$ is regular near the quasi-diagonal in $\SS^+\times\ov{\SS^+}$
provided $\NN^{22}=0$.
Recall that the number of odd variables on $\SS^+$ is $2g-2$. Thus, $\NN^{22}=0$ as long
as $2g-2\le 20$, i.e., $g\le 11$. This finishes the proof.
\qed



\subsection{The pole of the superperiod map along the divisor}

Let $S$ be a superscheme, $j:U\hra S$ a complement to a Cartier divisor
$D_0\sub S_{\red}$,
so $D_0$ is locally given by an equation $f=0$, where $f$ is a non-zero-divisor.
Let us assume in addition that each local equation of $D_0$ is 
a non-zero-divisor in $\bigoplus_{n\ge 0}\NN^n/\NN^{n+1}$,
where $\NN\sub \OO_S$ is the nilradical (this condition is automatic in the case of a smooth superscheme).
This implies that the natural homomorphism $\OO_S\to j_*\OO_U$ is injective. 

Now assume that $D_0$ is extended to a Cartier divisor $D\sub S$ (locally given by an even equation,
reducing to an equation of $D_0$). 
Let us define the sheaf of rings on $S$, $\AA_D\sub j_*\OO_U$ by
$$\AA_D:=\OO_S+\NN^2(D)+\NN^4(2D)+\ldots.$$
Note that locally any extension $\wt{f}\in\OO_S$ of an equation $f$ of $D_0$
is a non-zero-divisor in $\OO_S$, due to our assumption on $f$, so it gives a Cartier divisor on $S$ extending 
$D_0$. We will show below that in fact $\AA_D$ depends only on $D_0$.

Note that $\AA_D$ is a coherent sheaf of $\OO_S$-algebras, so it defines a superscheme $S_D$ with
a finite morphism to $S$, so that the embedding of $U$ into $S$ factors as the composition
$$U\hra S_D\to S.$$
Furthermore, we have $(S_D)_{\red}=S_{\red}$.


Let us also set
$$\II_D:=\NN\AA_D=\NN+\NN^3(D)+\NN^5(2D)+\ldots\sub\AA_D.$$
It is clear that $\II_D$ is an ideal in $\AA_D$. 

\begin{lemma} (i) The sheaves $\AA_D$ and $\II_D$ depend only on the reduced divisor $D^{red}\sub S^{red}$.

\noindent (ii)
There is a natural isomorphism of sheaves of rings
$$\ov{\AA}_D:=\AA_D/\II_D\simeq \OO_S/\NN\oplus\bigoplus_{n\ge 1}\NN^{2n}/\NN^{2n+1}(nD_0)|_{D_0}.$$
Thus, the closed sub-superscheme $\ov{S}_D\sub S_D$, associated with $\II_D$, is equipped with a natural
$\G_m$-action such that the invariant locus is precisely $S_{\red}$.
\end{lemma}

\begin{proof}
(i) We need to check that the sheaf $\OO_S+\frac{\NN^2}{f}+\frac{\NN^4}{f^2}+\ldots$ does not change if
we replace the local equation $f$ of $D$ by $f+n_2$ where $n_2\in \NN^2$. But this follows easily
from the formula
$$\frac{1}{f+n_2}=\frac{1}{f}(1-\frac{n_2}{f}+\frac{n_2^2}{f^2}-\ldots).$$

\noindent
(ii) First, let us check that the map
\begin{equation}\label{AD/ID-isom-map-eq}
\begin{array}{l}
\AA_D/\II_D\to \OO_S/\NN\oplus\bigoplus_{n\ge 1}\NN^{2n}/(\NN^{2n+1}+\NN^{2n}(-D)):\\
a_0+\frac{a_2}{f}+\frac{a_4}{f^2}+\ldots\mapsto (a_0\mod\NN,(a_{2n}\mod \NN^{2n+1}+\NN^{2n}(-D))_{n\ge 1})
\end{array}
\end{equation}
is well defined (here $a_{2i}\in\NN^{2i}$ and $f$ is a local equation of $D$). Indeed, suppose 
$$a_0+\frac{a_2}{f}+\frac{a_4}{f^2}+\ldots+\frac{a_{2n}}{f^n}=a_1+\frac{a_3}{f}+\ldots+\frac{a_{2n+1}}{f^n}\in \II_D,$$
where $a_i\in\NN^i$. 
We need to check that $a_0\in \NN$ and 
$a_{2i}\in \NN^{2i+1}+f\NN^{2i}$ for $i\ge 1$. 

Note that since the multiplication by $f$ is injective on all the quotients $\NN^i/\NN^{i+1}$, for $i\ge 0$,
it is injective on $\OO_S/\NN^i$. Hence, for $a\in \OO$, if $f^ma\in \NN^i$ then $a\in \NN^i$ (for $m\ge 0$).
We have $f^na_0\in\NN$, hence $a_0\in\NN$. Next, for every $i\ge 1$, we have
$$f^{2n-i}[f^i(a_0-a_1)+f^{i-1}(a_2-a_3)+\ldots+f(a_{2i-2}-a_{2i-1})+a_{2i}]\in \NN^{2i+1},$$
hence, we get that $fa+a_{2i}\in \NN^{2i+1}$ for some $a\in \OO$. But this implies that $fa\in\NN^{2i}$, so $a\in\NN^{2i}$,
therefore, $a_{2i}\in \NN^{2i+1}+f\NN^{2i}$, as claimed.

Thus, the map \eqref{AD/ID-isom-map-eq} is well defined. It is clearly surjective. Now suppose
$a\in \AA_D$ gives an element in its kernel. Then we have
$$a=a_1+\frac{a_3+fb_2}{f}+\frac{a_5+fb_4}{f^2}+\ldots=
(a_1+b_2)+\frac{a_3+b_4}{f}+\ldots$$
which is in $\II_D$, so $a\equiv 0\mod\II_D$.
\end{proof}


\begin{definition} For a Cartier divisor $D_0\sub S_{\red}$, whose local equations are non-zero-divisors
in $\bigoplus_{i\ge 0}\NN^i/\NN^{i+1}$,
we define the sheaf of rings on $S$,
$\AA_{D_0}\sub j_*\OO_U$, by choosing locally an extension of $D_0$ to a Cartier
divisor $D\sub S$ and setting
$$\AA_{D_0}:=\AA_D.$$
We also have a well defined ideal $\II_{D_0}\sub \AA_{D_0}$, given locally by $\II_D\sub \AA_D$.
We denote by $S_{D_0}$ the superscheme over $S$ with the structure sheaf $\AA_{D_0}$,
and by $\ov{S}_{D_0}\sub S_{D_0}$ the sub-superscheme associated with $\II_{D_0}$.
\end{definition}

Now assume that $f:U\to X$ is a morphism of superschemes, 
such that corresponding morphism $f_0:U_{\red}\to X$ extends
to a regular morphism $f_{\red}:S_{\red}\to X$. 

\begin{definition}
We say that $f$ has {\it regular pole} along $D_0$ if locally the pull-backs of
regular functions from $X$ belong to $\AA_{D_0}$. Equivalently, we require that $f$ extends to a regular morphism
$\wt{f}:S_{D_0}\to X$.
\end{definition}

Note that if $f$ has a regular pole along $D_0$ then we have the residual map
$$\Res_{D_0}(f):\ov{S}_{D_0} \to X$$
obtained as the composition 
$$\ov{S}_{D_0}\to S_{D_0}\rTo{\wt{f}} X.$$
Furthermore, the restriction of $\Res_{D_0}(f)$ to $S_{\red}$ is $f_{\red}$.

Now let $\SS^+$ denote the even component of the moduli of supercurves, and let 
$\UU\sub \SS^+$ be the open substack corresponding to theta-characteristics with vanishing cohomology.
Then we have a well defined superperiod map
\begin{equation}\label{superperiod-map-LG-eq}
\per:\wt{\UU}\to \LL\GG,
\end{equation}
where $\LL\GG$ is the Lagrangian Grassmannian of a $2g$-dimensional symplectic space over $\C$,
$\wt{\UU}\to \UU$ is the covering corresponding to a choice of a symplectic basis in cohomology.

Our results show that the superperiod map has a regular pole along the divisor $\DD_0$ where the theta-characteristic
has nontrivial global sections.

\begin{prop}\label{regular-pole-prop} 
The superperiod map \eqref{superperiod-map-LG-eq} has a regular pole
along the preimage of the divisor $\DD_0\sub \SS_{\bos}^+$.
\end{prop}

\begin{proof} This follows from the proof of Theorem \ref{s5-thm}.
We have seen in that proof that locally near a point on $\DD_0$, the superperiod map is given by
a symmetric matrix $\Omega$ with entries in $\NN^2$ such that $f\Omega$ is regular, where
$f$ reduces to an equation of $\DD_0$.
\end{proof}

\begin{remark} In the case when $S$ is smooth and the divisor $D_0\sub S_{\red}$ is smooth, we can
define a superorbifold $\wt{S}_{D_0}$ such that $S_{D_0}$ is the coarse moduli space of $\wt{S}_D$. 
Namely, let $f$ be a local function on $S$ such that $f\mod\NN$ defines $D_0$.
Let us consider the double covering $Z_D\to S$ given by $t^2=f$, and let $\wt{Z}_D$ be the blow-up of
$Z_D$ at the ideal $(t)+\NN$. If $(x_1,x_2,\ldots,x_n,(\th_j))$ are local (even and odd) coordinates on $S$, such that $x_1=f$ then 
$(t,x_2,\ldots,x_n,(\frac{\th_j}{t}))$ are local coordinates on $\wt{Z}_D$, so $\wt{Z}_D$ is a supermanifold. 
We have a natural action
of $\Z/2$ on $\wt{Z}_D$, sending $t$ to $-t$, and we consider the quotient-stack $\wt{Z}_D/(\Z/2)$.
One can check that the resulting stacks glue into a global stack $\wt{S}_{D_0}$ depending only on $D_0$.
Thus, away from the singular locus of $\DD$, the superperiod map extends to a regular map from a superorbifold over
$\SS^+$.
\end{remark}

\section{Second variation of the period map as a Massey product}\label{s-6}

\subsection{Second variation map}

Let $X$ be a smooth supervariety. Let us denote by 
$$\Om'_X\sub \Om_X$$ the $\OO_X$-submodule
generated by all $df$, where $f$ is an even function. Locally, $\Om'_X$ is generated by $dx_i$ and $d(\th_i\th_j)$
for $i<j$. The restriction $\Om'_X|_{X_{\red}}$ is locally free and fits into an exact sequence
$$0\to \NN^2/\NN^3\to \Om'_X|_{X_{\red}}\to \Om_{X_{\red}}\to 0.$$
Note that if $x\in X_{\red}$ is a point, $\fm\sub \OO_{X,x}$ the corresponding maximal ideal,
then the fiber of $\Om'_X|_{X_{\red}}$ at $x$ can be identified with $\fm^{ev}/(\fm^{ev})^2$,
where $\fm^{ev}$ is the even part of $\fm$.

Now assume that we have a morphism $f:X\to Y$,
where $Y$ is reduced (hence, purely even). Let $f_{\red}=f|_{X_{\red}}$ be the corresponding
map $X_{\red}\to Y$.
Then
we have the induced map $f^*\Om_Y\to \Om'_X$, and hence,
\begin{equation}\label{2nd-variation-dual-eq}
(f_{\red})^*\Om_Y\to \Om'_X|_{X_{\red}}.
\end{equation}
Furthermore, its composition with the projection to $\Om_{X_{\red}}$ is exactly the map induced by $df_{\red}$.
Dually, we get a diagram
\begin{diagram}
0&\rTo{}&T_{X^\red}&\rTo{}& (\Om'_X|_{X^\red})^\vee&\rTo{}& (\NN^2/\NN^3)^\vee&\rTo{}& 0\\
&&                           &\rdTo{df_{\red}}&\dTo{}\\
&&&&(f_{\red})^*T_Y
\end{diagram}
It follows that there is a unique map of $\OO_{X^\red}$-modules,
$$d^{(2)}f:{\bigwedge}^2(\NN/\NN^2)^\vee\simeq  (\NN^2/\NN^3)^\vee\to \coker(df_{\red})$$
induced by the vertical arrow in the above diagram. We call $d^{(2)}f$ the {\it second variation map}.

Given a point $x\in X_{\red}$, to compute $d^{(2)}f$ at $x$, we observe that
the map \eqref{2nd-variation-dual-eq} at $x$ is given by the natural map
\begin{equation}\label{2nd-variation-pointwise-eq}
\fm_y/\fm_y^2\to \fm_x^{ev}/(\fm_x^{ev})^2.
\end{equation}
Note that the latter map is well defined also when $X$ is replaced by its sub-superscheme.
Thus, we can choose coordinates $(x_1,\ldots,x_n;\th_1,\ldots,\th_m)$ near $x$ and consider the closed subscheme
$$X^{odd,\le 2}(x)\sub X$$
given by the ideal generated $(x_1,\ldots,x_n)+\fm_x^3$. 
Note that the algebra of functions on $X^{odd,\le 2}(x)$ is the truncated exterior algebra 
$\bigwedge(\th_1,\ldots,\th_m)/\bigwedge^{\ge 3}$. 
Let $\ov{\fm}_x$ denote the maximal ideal in this algebra.
Then it is easy to see that the composition
$$(\NN^2/\NN^3)|_x\simeq \NN^2/(\NN^3+\fm\NN^2)\to  \fm_x^{ev}/(\fm_x^{ev})^2\to \ov{\fm}_x^{ev}$$
is an isomorphism. Thus, the dual of the map $d^{(2)}f$ at $x$ can be identified with the restriction 
to $\ker((df^\red)^*)$ of the map
$$\fm_y/\fm_y^2\to \ov{\fm}_x^{ev},$$
induced by the map $X^{odd,\le 2}(x)\to X\to Y$.

\begin{remark} It is clear that the map \eqref{2nd-variation-pointwise-eq} can also be interpreted
as the differential of the induced map of usual schemes
$$X/\Ga\to Y,$$
where $X/\Ga$ is the scheme $(X,\OO_X^{ev})$ called the {\it bosonic quotient} of $X$ in \cite{CV}.
\end{remark}

\subsection{Second variation of a map to the Grassmannian}

Assume that we have an exact sequence of vector bundles of even rank over a supervariety $X$,
$$0\to \SS\to V\ot\OO_X\to \QQ\to 0,$$
so that we have an associated morphism $f:X\to \Gr(V)$ to the Grassmannian of $V$.

Given a point $x\in X$, we have the induced embedding of the fibers at $x$,
$\SS|_x\hra V$. Now the composition
$$\SS|_x\hra V\hra V\ot\OO_X\to \QQ$$ 
factors through a map
$$\SS|_x\to \fm_x \QQ.$$
The induced map
$$\SS|_x\to \fm_x \QQ/\fm_x^2\QQ\simeq (\fm_x/\fm_x^2)\ot \QQ|_x$$
corresponds to a map
$$T_x X\to \Hom(\SS|_x,\QQ|_x)$$
which is precisely the tangent map to $f$ at $x$.

To calculate $d^{(2)}f_x$, we apply a similar to procedure to the restriction 
$$0\to \ov{\SS}\to V\ot \OO_{X^{odd,\le 2}(x)}\to \ov{\QQ}\to 0$$
of our sequence of vector bundles
to a sub-superscheme $X^{odd,\le 2}(x)$ (which depends on a choice of coordinates near $x$).
Namely, as above, from the sequence of bundles on $X^{odd,\le 2}(x)$ we get a map
$$\SS|_x\hra V\to (\ov{\fm}_x\ov{\QQ})^{ev}\simeq \ov{\fm}^{ev}_x\ot \QQ|_x.$$
Now $d^{(2)}f_x$ is the composition 
of the corresponding map 
$$(\NN^2/\NN^3)|_x^\vee\simeq (\ov{\fm}^{ev}_x)^\vee\to\Hom(\SS|_x,\QQ|_x)$$
with the projection to $\coker((df_{\red})_x)$. 

Note that the analog of the map $d^{(2)}f_x$ makes sense if we start with an exact sequence
$$0\to \SS\to \VV\to \QQ\to 0$$
of bundles on a supermanifold, where $\VV$ is equipped with an integrable connection.
Namely, the connection gives a trivialization of $\VV/\fm_x^3\VV$, and then we can apply the
same construction.

\subsection{Second variation of the period map}

Let $\pi:X\to S=\SS^+$ be the universal supercurve over the even component of the moduli space of supercurves.
Then over the open subset corresponding to theta-characteristics with no global sections, the exact sequence 
\eqref{const-sh-resolution}
induces an exact sequence of bundles over $S$,
\begin{equation}\label{period-ex-seq}
0\to \pi_*\om_{X/S}\to R^1\pi_*\C_{X/S}\to R^1\pi_*\OO_X\to 0
\end{equation}
Furthermore, the bundle in the middle carries a Gauss-Manin connection, so we have the corresponding
period map, $\per:\wt{S}\to \Gr$, from a covering of $S$ to the Grassmannian.
We would like to calculate the corresponding second variation map. As we observed above, the second variation
map depends only on the flat connection on $R^1\pi_*\C_{X/S}$, so it can be calculated over $S$.

\begin{theorem}
Let $C$ be a smooth projective curve of genus $g$, $L$ a theta-characteristic on $C$ with
$H^0(C,L)=0$. Let $[C,L]$ be the corresponding point of $S$.

(i) The tangent map 
\begin{equation}\label{kappa-def-eq}
\kappa=(d\per_{\red})_{[C,L]}: H^1(C,\om_C^{-1})\to \Hom(H^0(C,\om_C),H^1(C,\OO_C))
\end{equation}
is given by the cup-product $H^1(C,\om_C^{-1})\ot H^0(C,\om_C)\to H^1(C,\OO_C)$.

\noindent
(ii) The negative of the second variation map,
$$-d^{(2)}\per_{[C,L]}: {\bigwedge}^2 H^1(C,L^{-1})\to \Hom(H^0(C,\om_C),H^1(C,\OO_C))/\im(\kappa),$$
is given by the skew-symmetrization of the Massey product map
\begin{equation}\label{main-MP-eq}
H^1(C,L^{-1})\ot H^1(C,L^{-1})\to  \Hom(H^0(\om_C),H^1(\OO_C))/\im(\kappa)
\end{equation}
sending $x_1\ot x_2$ to the map
$$\alpha \mapsto d^{-1}(x_1\cdot x_2)\cdot \alpha+x_1\cdot d^{-1}(x_2\cdot \alpha)$$
viewed modulo the image of $\kappa$ (where $d$ is the differential in an appropriate dg-model).
\end{theorem}

\begin{proof}
For any family of supercurves $X/S$, where $S$ is affine, 
with the underlying usual curve $\CC/S_{\red}$, and
a marked point $p:S_{\red}\to \CC$, we can consider the covering of $X$ by $\UU_1=X\setminus p(S^\red)$
and $\UU_0$, the formal neighborhood of $p(S^\red)$. For brevity, we will denote $p(S^\red)$ simply as $p$.
The main idea is to realize explicitly the exact sequence 
\eqref{period-ex-seq} using the covering $(\UU_0,\UU_1)$.
In addition, we will use an explicit realization of the universal curve over the subscheme $\SS^{odd,\le 2}([C,L])$.

We begin by representing $R\pi_*\OO_X$ by the complex
$$\PP^0\rTo{d_{\OO_X}} \PP^1$$
where $\PP^0=\OO_X(\UU_1)$, $\PP^1=\OO_X(\UU_0\setminus p)/\OO(\UU_0)$, and $d_{\OO_X}$
is induced by the restriction from $\UU_1$ to $\UU_0\setminus p=\UU_0\cap \UU_1$.

On the other hand, using the embedding $j:\UU_1\hra \CC$, we obtain quasi-isomorphisms 
$$\C_{X/S}\to [\OO_X\to \om_{X/S}]\to [j_*\OO_{\UU_1}\to (j_*\om_{\UU_1/S})^{ex}],$$
where
$(j_*\om_{\UU_1/S})^{ex}\sub j_*\om_{\UU_1/S}$ denotes the subsheaf of forms that are exact near $p$
(in other words, this is the sheaf image of $\de$).
Note that the latter subsheaf is not acyclic for $\pi$. However, since $j_*\OO_{\UU_1}$ is $\pi$-acyclic, we can still 
represent $\tau_{\le 1}R\pi_*\C_{X/S}$ by the two-term complex
$$\PP^0\rTo{\de} \QQ^1$$
where $\QQ^1=\om_{X/S}(\UU_1)^{ex}$ is the space of forms on $\UU_1$, exact near $p$.

Now the exact sequence \eqref{period-ex-seq} can be represented explicitly as
$$0\to \om_{X/S}(X)\to \coker(\PP^0\rTo{\de} \QQ^1)\to \coker(\PP^0\rTo{d_{\OO_X}} \PP^1)\to 0$$
where the first map sends $\a\in\om_{X/S}(X)$ to $\a|_{\UU_1}\in \QQ^1$, while the second map
is induced by the map 
$$\QQ^1\to \PP^1: \b\mapsto\de^{-1}(\b|_{\UU_0\setminus p}).$$

Now we specialize to the case when $S_{\red}$ is the point $[C,L]$, and our family $X$ over $S$  
is obtained by gluing trivial families of supercurves $S\times X_0$ and $S\times X_1$. Here
$$X_i=(U_i,(\OO\oplus L)|_{U_i}), \ i=0,1,$$
where $U_0$ is the formal neighborhood of a point $p\in C$ and $U_1=C\setminus p$.
The gluing is given by an automorphism $T$ of $S\times X_{01}$, where $X_{01}=(U_{01},(\OO\oplus L)|_{U_{01}})$.
In this case we have identifications 
\begin{align*}
&\PP^0=P^0:=\OO_{X_1}(U_1)\otimes\OO_S, \ \ \PP^1=P^1=\OO_{X_0}(U_0\setminus p)/\OO_{X_0}(U_0)\otimes \OO_S, \\ 
&\QQ^1=Q^1=\om_{X_1}(U_1)^{ex}\otimes\OO_S,\\
&\om_{X/S}(X)=\{\a_0\in \om_{X_0}(U_0)\ot\OO_S, \a_1\in \om_{X_1}(U_1)\ot\OO_S \ |\ T^*\a_1|_{U_0\setminus p}=\a_0|_{U_0\setminus p}\}\\
&=\{\a_1\in \om_{X_1}(U_1)\ot\OO_S\ |\ T^*\a_1|_{U_0\setminus p} \text{ is regular at } p\}.
\end{align*}
Under these identifications the differential $d_{\OO_X}:\PP^0\to \PP^1$ corresponds to
$$d_T: P^0\to P^1: f\mapsto T^*(f|_{U_0\setminus p}),$$
$\de:\PP^0\to \QQ^1$ corresponds to $\de:P^0\to Q^1$, 
still just induced by $\de:\OO_{X_1}(U_1)\to \om_C(U_1)$, the map $\QQ^1\to \PP^1$ corresponds to the map
\begin{equation}\label{main-period-differential}
Q^1\to P^1: \b\mapsto T^*\de^{-1}(\b|_{U_0\setminus p}).
\end{equation}
Finally, the embedding of $\om_{X/S}(X)$ into $\coker(P^0\rTo{\de} Q^1)$ sends $(\a_0,\a_1)$ to the image of $\a_1$.

The Gauss-Manin connection on $\coker(P^0\to Q^1)$ is induced by the connection on $Q^1$, such that
$\om_{X_1}(U_1)^{ex}$ are horizontal sections.

For the computation of the tangent map in (i), it is enough to consider the family over $\C[\eps]/(\eps^2)$
associated with $v\in H^1(C,T_C)$. Namely, we realize $v$ by a vector field on $U_{01}=U_0\setminus p$, 
and consider the automorphism $T$ on $(U_0\setminus p)[\eps]$ given by this vector field. 
Now given $\a\in H^0(C,\om_C)$, we consider $\a|_{U_1}$ as a horizontal section of $\coker(P^0\rTo{\de} Q^1)$,
and apply the map \eqref{main-period-differential} to it.
Note that $T^*$ is given by the Lie derivative $L_v=\de i_v$. Thus, 
$$T^*\de^{-1}(\a|_{U_0\setminus p})=\lan v,\a|_{U_0\setminus p}\ran,$$
which corresponds to taking the cup product of the class of $v$ in $H^1(C,T_C)$ with $\a\in H^0(C,\om_C)$.

For the computation of the second variation we take $S=\Spec(\bigwedge W/{\bigwedge}^{\ge 3}W)$, where
$$W:=H^1(C,L^{-1})^\vee.$$ 
Let $\varphi\in H^1(C\times S, L^{-1})$ denote the universal section given as
$$\varphi=\sum_i \eta_i\cdot b_i\in H^1(C\times S, L^{-1})=H^1(C,L^{-1})\ot \bigwedge W/{\bigwedge}^{\ge 3}W,$$
where $(\eta_i)$ is a basis of $H^1(C,L^{-1})$ and $(b_i)$ is the dual basis of $W$.
We realize $\varphi$ by a section of $L^{-1}$ over $U_{01}\times S=(U_0\setminus p)\times S$.

The automorphism $T^*$ of $(\OO_C\oplus L)\boxtimes \OO_S$ over $(U_0\setminus p)\times S$ is given by
$$(g,\psi)\mapsto (g+\varphi\psi,\psi+\varphi dg+\frac{1}{2}Q(\varphi)\psi),$$
where $g\in \OO_{C\times S}$, $\psi\in L\boxtimes \OO_S$, and $Q$ is a certain quadratic form (see Section 
\ref{s-coord}).
Using this explicit form of $T$ we can get a different presentation of $R^1\pi_*\OO_X=\coker(P^0\rTo{d_T} P^1)$.
Note that the elements of $P^0$ are pairs $(g,\psi)$, where $g\in \OO(U_1)\otimes\OO_S$, 
$\psi\in L(U_1)\otimes\OO_S$, and 
$$d_T(g,\psi)=T^*(g|_{U_0\setminus p},\psi|_{U_0\setminus p})\mod (\OO(U_0)\oplus L(U_0))\ot\OO_S.$$
Furthermore, by the assumption on the vanishing of $H^*(C,L)$, the differential
$$d_L: L(U_1)\otimes\OO_S\to (L(U_0\setminus p)/L(U_0))\ot \OO_S$$
(induced by the restriction to $U_0\setminus p$) is invertible. It follows that for any $\psi\in (L(U_0\setminus p)/L(U_0))\ot\OO_S$, we have
$$d_T(0,d_L^{-1}\psi)=(\varphi (d_L^{-1}\psi)|_{U_0\setminus p}, \psi+\frac{1}{2}Q(\varphi)(d_L^{-1}\psi)|_{U_0\setminus p}),$$
so that
$$(0,\psi)\equiv -(\varphi (d_L^{-1}\psi)|_{U_0\setminus p},0)-(0,\frac{1}{2}Q(\varphi)(d_L^{-1}\psi)|_{U_0\setminus p}) \mod \im(d_T).$$
We can iterate this procedure by applying the same identity to the second term in the right-hand side, etc. As a result, we get
$$(0,\psi)\equiv (A_\varphi ((d_L^{-1}\psi)|_{U_0\setminus p}),0) \mod \im(d_T),$$
where $A_\varphi$ is a morphism of sheaves $L\to\OO$, $A_\varphi(\psi)=-\varphi\psi+\ldots$
over $(U_0\setminus p)\times S$.
This leads to an isomorphism 
\begin{align}\label{coker-isom-eq}
&\coker(d_T)\rTo{\sim}
\coker[(\id+A_\varphi(\varphi\cdot d))\circ d_\OO:
\OO(U_1)\otimes \OO_S\to (\OO(U_0\setminus p)/\OO(U_0))\ot\OO_S]:  \nonumber\\
& (g,\psi)\mapsto g+A_\varphi ((d_L^{-1}\psi)|_{U_0\setminus p}).
\end{align}

Now to calculate $d^{(2)}\per_{[C,L]}$, we start with a global form $\a\in H^0(C,\om_C)$.
Then we consider the restriction $\a|_{U_1}$ as the corresponding horizontal section of $\coker(\de:P^0\to Q^1)$.
Then, applying the map induced by \eqref{main-period-differential} to this section we get
$$T^*(d^{-1}(\a|_{U_0\setminus p}),0)=(d^{-1}\a_{U_0\setminus p},\varphi \a|_{U_0\setminus p})=
(0,\varphi \a|_{U_0\setminus p})$$
since $d^{-1}\a_{U_0\setminus p}$ extends to $U_0$, and so is zero in $P^1$.
Finally, we apply isomorphism \eqref{coker-isom-eq} and get the element in
$\coker[(\id+A_\varphi(\varphi\cdot d))\circ d_\OO]$ represented by
$$A_\varphi([d_L^{-1}(\varphi \a|_{U_0\setminus p})]|_{U_0\setminus p})
=-\varphi\cdot [d_L^{-1}(\varphi \a|_{U_0\setminus p})]|_{U_0\setminus p}$$
since we quotient by $\bigwedge^{\ge 3}$ in $\OO_S$.

Now using the formula for $\varphi$ leads to the asserted expression in terms of the Massey product.
\end{proof}

\begin{remark}
Note that the usual Massey product
$$H^1(L^{-1})\ot H^1(L^{-1})\ot H^0(\om_C)\dashedrightarrow{} H^1(\OO_C)$$
is ill-defined since the composition $H^1(\om_C^{-1})\rTo{\cdot \alpha} H^1(\OO_C)$ is surjective
for every $\alpha\neq 0$. However, we consider instead a well defined map
$$H^1(L^{-1})\ot H^1(L^{-1})\to \Hom(H^0(\om_C),H^1(\OO_C))/\im(\kappa).$$
\end{remark}


Next, we are going to calculate the Massey product map \eqref{main-MP-eq}.

\section{Study of the Massey product}\label{s-7}

\subsection{Relation to a univalued Massey product}

First, we observe that the Massey product \eqref{main-MP-eq} can be given as
$$\alpha\mapsto m_3(x_1,x_2,\alpha),$$
where we use the minimal $A_\infty$-structure obtained by the homological perturbation (this follows from
\cite[Prop.\ 1.1]{P-CYBE}).
Next, using the fact that this $A_\infty$-structure can be chosen to be cyclic with respect to the Serre duality pairing,
we see that 
$$\langle m_3(x_1,x_2,\alpha_1), \alpha_2\rangle=\pm \langle x_1, m_3(x_2,\alpha_1,\alpha_2)\rangle,$$
for $x_i\in H^1(L^{-1})$, $\alpha_i\in H^0(\om_C)$. 

Using Serre duality we can view the map $\kappa$ (see \eqref{kappa-def-eq}) as the map
$$H^0(\om_C^{\ot 2})^*\simeq H^1(\om_C^{-1})\to H^0(\om_C)^*\ot H^0(\om_C)^*,$$
dual to the multiplication map
$$H^0(\om_C)\ot H^0(\om_C)\to H^0(\om_C^{\ot 2}).$$
Thus, the above cyclicity implies that the dual to the Massey product map \eqref{main-MP-eq}
can be identified with the map 
$$K\to \Hom(H^1(L^{-1}),H^0(\om\ot L)): \sum_i \alpha_i\ot\alpha'_i\mapsto (x\mapsto \sum_i m_3(x,\alpha_i,\alpha'_i)),$$
where $K\sub H^0(\om_C)\ot H^0(\om_C)$ is the kernel of the multiplication map.

Next, applying the $A_\infty$-identity to the elements 
$\alpha,x,\sum_i\alpha_i\ot\alpha'_i$, where $\alpha\in H^0(\om_C)$, we get
\begin{equation}\label{two-m3-eq}
\alpha\cdot \sum_i m_3(x,\alpha_i,\alpha'_i)=\sum_i m_3(\alpha,x,\alpha_i)\cdot\alpha'_i.
\end{equation}
Note that if we know the left-hand side for all $\alpha\in H^0(\om_C)$, then this determines the element
$\sum_i m_3(x,\alpha_i,\alpha'_i)\in H^0(\om_C\ot L)$, since the map
$$H^0(\om_C\ot L)=\Hom(\om_C, \om_C^2\ot L)\to \Hom(H^0(\om_C),H^0(\om_C^2\ot L)):y\mapsto (\a\mapsto \a\cdot y)$$
is injective for $g\ge 1$ (since $\om_C$ is generated by its global sections).

Thus, our Massey product is uniquely determined by the map
\begin{equation}\label{MP-uni-eq}
m_3: H^0(C,\omega_C)\otimes H^1(C,L^{-1})\otimes H^0(C,\omega_C)\to
  H^0(C,\omega_C\otimes L),
\end{equation}
which is a univalued Massey product due to the assumption $H^0(C,L)=H^1(C,L)=0$.

\subsection{Calculation of the Massey product via the triangulated structure} 

Let $\mathcal T$ be a triangulated category,
$\alpha\in\Hom(A,B),\beta\in\Hom(B,C[1]),\gamma\in\Hom(C,D)$
composable morphisms
\[
  A\stackrel{\alpha}\longrightarrow B\stackrel{\beta}\longrightarrow
  C[1]\stackrel{\gamma[1]}\longrightarrow D[1]
\]
such that $\beta\circ\alpha=0$ and $\gamma[1]\circ\beta=0$. Then a
Massey product
\[
  \operatorname{MP}(\alpha,\beta,\gamma)\in
  \Hom(A,D)/(\gamma\circ\Hom(A,C)+\Hom(B,D)\circ\alpha)
\]
is defined as follows. The map $\beta$ is part of an exact triangle
\[
  C\stackrel{f}\longrightarrow V\stackrel{g}\longrightarrow
  B\stackrel{\beta}\longrightarrow C[1].
\]
With our assumptions on the compositions of $\alpha,\beta,\gamma$
there exist morphisms $\tilde\alpha\colon A\to V$ and
$\tilde\gamma\colon V\to D$ such that $g\circ\tilde\alpha=\alpha$ and
$\tilde\gamma\circ f=\gamma$. The Massey product is
$\tilde\gamma\circ\tilde\alpha\in\Hom(A,D)$. Any two choices of
$\tilde\alpha$ differ by $f\circ \sigma$ with $\sigma \in \Hom(A,C)$;
the corresponding Massey products differ by $\gamma\circ
\sigma$. Similarly, changing the choice of $\tilde\gamma$ changes the
Massey product by $\tau\circ\alpha$ for some $\tau\in \Hom(B,D)$. Thus,
the Massey product is a well-defined element of the quotient of
$\Hom(A,D)$ by these ambiguities.  A particularly simple situation,
which occurs in our case, is when $\Hom(A,C)$ and $\Hom(B,D)$
vanish. In this case, the compositions always vanish and there is no
ambiguity: the Massey product is a well-defined map
\[
  \operatorname{MP}\colon
  \Hom(A,B)\otimes\Hom(B,C[1])\otimes\Hom(C,D)\to\Hom(A,D).
\]

Now we can apply the above recipe to calculate the Massey product
\eqref{MP-uni-eq}, associated with $(C,L)$, where $H^*(C,L)=0$. Note that the compatibility of the Massey products
calculated using a dg-enhancement and using a triangulated structure is well known and goes back
to \cite[Sec.\ 5.A]{BK} (see \cite[Sec.\ 3.2]{FP} for details).
Using this compatibility, we can rewrite \eqref{MP-uni-eq} as
\[
  \operatorname{MP}\colon \Hom(\mathcal O_C,\omega_C)\otimes
  \Ext^1(\omega_C,L)\otimes \Hom(L,\omega_C\otimes L)\to\Hom(\mathcal
  O_C,\omega_C\otimes L).
\]
Let $V(\beta)$ be an extension
\[
  0\to L\to V(\beta)\to \omega_C\to 0
\]
corresponding to $\beta\in\Ext^1(\omega_C,L)$. Then 
\begin{equation}\label{m3-MP-formula-eq}
 m_3(\alpha_1,\beta,\alpha_2)=\operatorname{MP}(\alpha_1,\beta,\alpha_2)(1)=\tilde\alpha_2\circ\tilde\alpha_1(1)
\end{equation}
where $\tilde\alpha_1$, $\tilde\alpha_2$ are the lifts of $\alpha_1,\alpha_2$
in the diagram
\begin{equation}\label{MP-curve-diagram}
  \begin{tikzcd}[column sep=small]
    \mathcal O_C
    \arrow[rr,"\alpha_1"]\arrow[rrrd,dashed,"\tilde\alpha_1" below]&&
    \omega_C \arrow[rr,"{[1]}"] & &
    L\arrow[rr,"\alpha_2"]\arrow{dl}& &\omega_C\otimes L\\
    & & &V(\beta)\arrow[ul]\arrow[urrr,dashed,"\tilde\alpha_2" below]
    & & &
  \end{tikzcd}
\end{equation}
in the derived category of coherent sheaves on $C$.

\subsection{Coordinates on $\Ext^1(\omega_C,L)$}
We are going introduce coordinates on the $(2g-2)$-dimen\-sional vector space
$\Ext^1(\omega_C,L)=H^1(C,L^{-1})$  (they will depend on a choice of a
generic global differential on $C$).

\begin{lemma}\label{theta-char-duality-lem}
  Let $L$ be a theta-characteristic on $C$ without non-trivial
  sections.  Let $\alpha\in H^0(C,\omega_C)$ be a nonzero differential
  with simple zeros $P_1,\dots, P_{2g-2}$.  Then the
  evaluation map
  \[
    \varphi_\alpha\colon H^0(C,L^3)\to \bigoplus_{i=1}^{2g-2}L^3|_{P_i}
  \]
  at the zeros of $\alpha$ is an
  isomorphism. Dually, we have an isomorphism
  \[
   \varphi_\alpha^\vee\colon\bigoplus_{i=1}^{2g-2}L|_{P_i}\to H^1(C,L^{-1}).
  \]
  The perfect pairing
  $$(\bigoplus_iL^3|_{P_i})\otimes(\bigoplus_iL|_{P_i})\to \mathbb C$$
  induced by the Serre duality is
  $$t\otimes s\mapsto \sum_{i} t_is_i/\alpha'(P_i)$$ 
  where
  $\alpha'(P_i)\in (\omega_C)^2|_{P_i}$ is the derivative of $\alpha$
  at the zero $P_i$.
\end{lemma}  

\begin{proof}
  The map $\varphi_\alpha$ appears in the long exact
  sequence associated with
  \[
    0\to L\stackrel{\alpha}\to L^3\to \bigoplus_i L^3|_{P_i}\to 0.
  \]
Since  $H^0(C,L)=H^1(C,L)=0$, $\varphi_\alpha$ is an
  isomorphism. Similarly, $\varphi^\vee_\alpha$ comes from boundary homomorphism
  associated with the exact
  sequence
  \begin{equation}\label{varphi-vee-extension-seq}
    0\to L^{-1}\stackrel{\alpha}\to L\to \bigoplus_i L|_{P_i}\to 0.
  \end{equation}
 Note that we have an isomorphism 
$$\bigoplus_i L^{-1}(P_i)|_{P_i}\rTo{\a'(P_i)} \bigoplus_i L|_{P_i},$$
whose composition with $\varphi^\vee_\alpha$ is the standard coboundary map
$$\varphi^\vee:\bigoplus_i L^{-1}(P_i)|_{P_i}\to H^1(C,L^{-1}).$$
It is well known that the Serre duality pairing of $\varphi^\vee(x)$ with $y\in H^0(C,L\om_C)$
is equal to the natural pairing of $x$ with $(y|_{P_i})$ (that uses trivializations of $\om_C(P_i)|_{P_i}$).
This implies our assertion.
\end{proof}

\subsection{Formula for the Massey product}


\begin{prop} \label{p-02} Fix a differential $\alpha\in H^0(C,\omega_C)$
  with a simple divisor of zeros $D=P_1+\ldots+P_{2g-2}$.
For $y\in \bigoplus_{i=1}^{2g-2} L|_{P_i}$, consider
  $\beta=\varphi^\vee_\alpha(y)\in\Ext^1(\omega,L)\cong
  H^1(C,L^{-1})$. Then
  \[
    m_3(\alpha_1,\beta,\alpha_2)=\frac{\varphi^{-1}_\alpha(y\cdot\alpha_1|_D)\alpha_2-\varphi^{-1}_\alpha(y\cdot\alpha_2|_D)\alpha_1}{\alpha}.
  \]
\end{prop}
\begin{proof}
  We first construct an extension $V(\beta)$ with the extension class
  $\beta=\varphi^\vee_\alpha(y)\in\Ext^1(\omega,L)$ in terms of the
  coordinates $y_i\in L|_{P_i}$.
  
  \begin{lemma}\label{l-02}
    Let $\beta=\varphi^\vee_{\alpha}(y)\in \Ext^1(\omega_C,L)$.  Let $V(\beta)$ be the
    subsheaf of $L^3\oplus \omega_C$ consisting of sections
    $(t,\sigma)$ such that $t(P_i)=y_i\sigma(P_i)$. Then
    \[
      0\to L\stackrel{(\alpha,0)}\longrightarrow
      V(\beta)\stackrel{p_2}\longrightarrow \omega_C\to 0,
    \]
    with $p_2(t,\sigma)=\sigma$, is an extension whose class is
    $\beta$.
  \end{lemma}
  
  \begin{proof} We have $V(\beta)=\tilde V(\beta)\otimes\omega_C$
    where $\tilde V(\beta)\subset L\oplus\mathcal O_C$ is the
    subsheaf of sections $(\lambda,f)$ such that
    $\lambda(P_i)=y_if(P_i)$ for all $i=1,\dots,2g-2$.  We have to
    show that
    \[
      0\to L^{-1}\stackrel{(\alpha,0)}\longrightarrow \tilde
      V(\beta)\stackrel{p_2}\longrightarrow \mathcal O_C\to 0
    \]
    is an extension with class $\beta\in H^1(C,L^{-1})$. For this we recall that $\beta$ as the image
    of $y$ under the coboundary map, which means that the corresponding extension of $\mathcal O_C$ by $L^{-1}$ is the pullback
    of the extension \eqref{varphi-vee-extension-seq} under the map $y:\mathcal O\to \bigoplus_i L|_{P_i}$. This immediately gives the result.
  \end{proof}
  
  We can now construct the lifts $\tilde\alpha_1,\tilde\alpha_2$ in diagram \eqref{MP-curve-diagram}: 
  $\tilde \alpha_1$ is of the form
  $$1\mapsto (t,\alpha_1)$$ 
  where $t\in H^0(C,L^3)$
  is determined by the condition that  $(t,\alpha_1)\in V(\beta)$,
  i.e., 
  $$t(P_i)=y_i\alpha_1(P_i), \ \ i=1,\ldots,2g-2,$$ or equivalently,
   $\varphi_\a(t)=y\cdot \a_1|_D$.
  The condition that
  $\tilde\alpha_2$ is a lift of $\alpha_2$ implies that
  $\tilde\alpha_2$ has the form
  \[
    (u,\sigma)\mapsto (u\alpha_2-s\sigma)\alpha^{-1}
  \]
  for some $s\in H^0(C,L^3)$ determined by the condition that this map
  is regular at $P_i$. The latter condition is equivalent to
  \[
    u(P_i)\alpha_2(P_i)=s(P_i)\sigma(P_i), \quad i=1,\dots,2g-2.
  \]
  Since $u(P_i)=y_i\sigma(P_i)$, we can take
  $s(P_i)=y_i\alpha_2(P_i)$.  Thus,
  $$\tilde\alpha_2\circ\tilde\alpha_1(1)=(t\alpha_2-s\alpha_1)/\alpha,$$
  where $s\in H^0(C,L^3)$ is defined by the condition
  $$s(P_i)=y_i\alpha_2(P_i), \ \  i=1,\dots,2g-2.$$ Equivalently,
  $\varphi_\a(s)=y\cdot \a_2|_D$. Thus, using \eqref{m3-MP-formula-eq} we get the required formula for 
  $m_3(\a_1,\b,\a_2)$.
\end{proof}


\begin{cor}\label{main-MP-cor}
For $\sum_i \alpha_i\otimes \alpha'_i\in K\subset H^0(\omega_C)^{\otimes 2}$ and $\beta=\varphi^\vee(y)\in H^1(L^{-1})$,
one has
$$\sum_i m_3(\beta,\alpha_i,\alpha'_i)=-\frac{\sum_i \varphi^{-1}_\alpha(y\cdot\alpha_i|_D)\cdot \alpha'_i}{\alpha}\in H^0(\omega_C\otimes L).$$
\end{cor}

\begin{proof} Combining \eqref{two-m3-eq} with the formula of Proposition \ref{p-02} we get
\begin{align*}
&\alpha'\cdot\sum_i m_3(\beta,\alpha_i,\alpha'_i)\cdot \alpha=\sum_i m_3(\alpha',\beta,\alpha_i)\cdot \alpha'_i\cdot \alpha=\\
&\sum_i \varphi^{-1}_\alpha(y\cdot\alpha'|_D)\alpha_i\alpha'_i-\sum_i \varphi^{-1}_\alpha(y\cdot\alpha_i|_D)\alpha'_i\alpha'=-\sum_i 
\varphi^{-1}_\alpha(y\cdot\alpha_i|_D)\alpha'_i\cdot \a'.
\end{align*}
\end{proof}

\subsection{Computation for hyperelliptic curves}\label{hyperell-comp-sec}

Assume now that $C$ is hyperelliptic with the double covering map $f:C\to {\mathbb P}^1$ ramified at the points
$p_0,\ldots,p_{2g+1}\in C$. 
We choose an even theta-characteristic on $C$ to be
$$L:=\OO(p_1+\ldots+p_g-p_0).$$
Note that $h^0(C,\OO(p_1+\ldots+p_g))=1$ (see \cite[Lem.\ 2.6.2]{FP}), hence $h^0(C,L)=0$. 

Let us choose a coordinate $t$ on ${\mathbb P}^1$ such that $t(p_0)=\infty$, and set $a_i=t(f(p_i))$.
Then we have a natural identification
$$\OO_C(-4p_0)\rTo{dt} f^*\omega_{\P^1}.$$
The affine curve $C\setminus\{p_0\}$ can be identified with the double cover of $\A^1$ given by 
$$x^2=\prod_{i=1}^{2g+1} (t-a_i).$$
Note that $x$, viewed as a rational function on $C$ has simple zeros at $p_1,\ldots,p_{2g+1}$ and a pole of order $2g+1$ at $p_0$. 

Let $\tau:C\to C$ be the hyperelliptic involution, so that $\tau^*(t)=t$, $\tau^*(x)=-x$.
Then we have an isomorphism of $\Z/2$-equivariant line bundles
$$\omega_C\simeq f^*\omega_{\P^1}(\sum_{i=0}^{2g+1} p_i)\rTo{x/dt} \OO_C((2g-2)p_0)\otimes\chi,$$
where $\chi$ is a nontrivial character of $\Z/2$.
The $\Z/2$-equivariant isomorphism of $L^2$ with $\omega_C\otimes\chi$ comes from this identification and 
from the isomorphism
$$L^{\ot 2}\simeq \OO_C(2p_1+\ldots+2p_g-2p_0)\rTo{F}\OO_C((2g-2)p_0),$$
where $F=(t-a_1)\ldots(t-a_g)$.

Now let us consider a section 
$$G=\prod_{j=1}^{g-1}(t-b_j)$$ 
of $\OO_C((2g-2)p_0)$, where $b_j\in\A^1$ are distinct and disjoint from $(a_i)$, and set
$$\alpha=\frac{G dt}{x}\in H^0(C,\omega_C).$$
Note that $\tau^*$ acts on $H^0(C,\omega_C)$ as $-1$. 
Let $D=\cup f^{-1}(b_j)$ be the divisor of zeros of $\alpha$. For each $j=1,\ldots,g-1$, we pick a point $q_j^+\in f^{-1}(b_j)$,
so that $D=\sum (q_j^+ + q_j^-)$, where $q_j^-=\tau(q_j^+)$.

Recall that the restriction map $\varphi_\alpha:H^0(\omega_C\otimes L)\to H^0((\omega_C\otimes L)|_D)$ is an isomorphism,
and by Serre duality, the dual map gives an isomorphism 
$\varphi^\vee_\alpha:H^0(L|_D)\to H^1(L^{-1})$.
We will use natural identifications 
$$L|_D\simeq \OO_D, \ \ \OO_D \rTo{dt/x} \omega_C|_D\otimes\chi.$$
Thus, we view $\varphi_\alpha$ as a map
$$H^0(\omega_C\otimes L)\rTo{\sim} H^0(\OO_D)\otimes\chi: \beta\mapsto \frac{\beta}{dt}\cdot x|_D.$$

Recall (see Lemma \ref{theta-char-duality-lem})
that the duality between $H^0(L|_D)$ and $H^0((\omega_C\otimes L)|_D)$ is induced by
the product, the identification of $L^{\ot 2}$ with $\omega_C$ and by the residue map
$$\omega_C^{\ot 2}|_D\to k: \gamma\mapsto \sum\Res(\frac{\gamma}{\alpha}).$$
Using the above trivializations we can identify the pairing between $H^0(L|_D)$ and $H^0((\omega_C\otimes L)|_D)$
with the composition
$$H^0(\OO_D)\otimes (H^0(\OO_D)\ot\chi) \rTo{\frac{dt}{x}} H^0(L^{\ot 2}\ot \omega_C|_D)\rTo{\frac{F}{x}dt}
H^0(\omega_C^{\ot 2}|_D)\rTo{\sum\Res \frac{?\cdot x}{Gdt}} k.$$
Thus, this pairing is given by
\begin{equation}\label{y-y'-pairing}
\langle y,y'\rangle=\sum_j \frac{F(b_j)}{x_jG'(b_j)}\bigl(y(q_j^+)y'(q_j^+)-y(q_j^-)y'(q_j^-)\bigr),
\end{equation}
where $x_j=x(q_j^+)$.

Set $V=H^0(\OO_D)$. We have a decomposition $V=V^+\oplus V^-$ into the eigenspaces with respect to $\Z/2$-action,
so that $V^\pm$ is spanned by 
$$e_j^\pm:=\de(q_j^+)\pm \de(q_j^-),$$
where $\de(q_j^\pm)\in H^0(\OO_D)$ takes value $1$ at $q_j^\pm$ and $0$ at all other points of $D$.
Note that $V^+$ and $V^-$ are isotropic with respect to $\langle ?,?\rangle$ and
$$\langle e_j^+,e_{j'}^-\rangle=\de_{j,j'}\cdot \frac{2F(b_j)}{x_jG'(b_j)}.$$
We have a commutative diagram
\begin{diagram}
H^0(\omega_C\ot L)&\rTo{\varphi_\alpha}& H^0(\omega_C\ot L|_D)\\
\dTo{x/dt}&&\dTo{x/dt} \\
H^0(\OO_C(p_1+\ldots+p_g+(2g-3)p_0))\ot\chi&\rTo{}& V\ot\chi
\end{diagram}
where the bottom arrow is the natural restriction map.
Let $G_j$, $i=j,\ldots,g-1$, be polynomials of degree $g-2$ in $t$ such that
$G_j(b_{j'})=\de_{j,j'}$. Then we have the following basis of $H^0(\OO_C(p_1+\ldots+p_g+(2g-3)p_0))$:
$$(G_j), j=1,\ldots,g-1, \ \ (\frac{G_j x}{F}), j=1,\ldots,g-1,$$
where the first $g-1$ elements are symmetric with respect to $\tau$, and the last $g-1$ elements are antisymmetric. 
Thus, we have
$$\varphi_\alpha(G_j\frac{dt}{x})=e_j^+\cdot (\frac{dt}{dx}|_D),$$
$$\varphi_\alpha(\frac{G_jx}{F}\cdot\frac{dt}{x})=\frac{x_j}{F(b_j)}e_j^-\cdot (\frac{dt}{dx}|_D).$$

The basis of $H^0(\om_C)$ is given by $t^m\frac{dt}{x}$, $m=0,\ldots,g-1$.
Now we can apply the formula of Corollary \ref{main-MP-cor} to some element 
$$\xi=\sum_i f_i(t)\ot g_i(t) \frac{dt^{\ot 2}}{x^2}\in K\sub H^0(\om_C)^{\ot 2},$$
where $\sum_i f_ig_i=0$.
We have
$$-m_3(\varphi^\vee_\alpha(e_j^\pm)\ot \xi)=
\frac{\sum_i \varphi^{-1}_\alpha(e_j^\pm f_i(b_j) dt/x)g_i(t)}{G}.$$
Hence,
$$-m_3(\varphi^\vee_\alpha(e_j^+)\ot \xi)=
\frac{\sum_i f_i(b_j)g_iG_j}{G}\cdot \frac{dt}{x}=\frac{\sum_i f_i(b_j)g_i}{G'(b_j)(t-b_j)}\cdot \frac{dt}{x},$$
$$-m_3(\varphi^\vee_\alpha(e_j^-)\ot \xi)=
\frac{F(b_j)\sum_i f_i(b_j)g_iG_jx}{x_jFG}\cdot \frac{dt}{x}=\frac{F(b_j)\sum_i f_i(b_j)g_ix}{x_jG'(b_j)(t-b_j)F}\cdot \frac{dt}{x},$$
where we used the identity $(t-b_j)G_j=G/G'(b_j)$.

For applications to the second variation of the superperiod map,
we need to compute the skew-symmetrization of the corresponding maps
$$m_3(?\ot \xi):H^1(L^{-1})\to H^0(\omega_C\otimes L).$$
Equivalently, we have to skew-symmetrize the maps
$$A_\xi=\varphi_\alpha m_3(\varphi^\vee_\alpha(?)\ot \xi):V\to V\ot\chi$$
with respect to the pairing \eqref{y-y'-pairing}.

Restricting the right-hand sides of the formulas above to $D$, we get
$$-A_\xi(e_j^+)=\frac{\sum_if_i(b_j)g'_i(b_j)}{G'(b_j)}e_j^+ 
+\sum_{j'\neq j}\frac{\sum_i f_i(b_j)g_i(b_{j'})}{G'(b_j)(b_{j'}-b_j)}e_{j'}^+,
$$
$$-A_\xi(e_j^-)=\frac{\sum_if_i(b_j)g'_i(b_j)}{G'(b_j)}e_j^- 
+\sum_{j'\neq j}\frac{x_{j'}F(b_j)\sum_i f_i(b_j)g_i(b_{j'})}{x_jF(b_{j'})G'(b_j)(b_{j'}-b_j)}e_{j'}^-,
$$

Thus, 
$$-\langle e_k^+,A_\xi^*(e_j^-)\rangle=-\langle A_\xi(e_k^+),e_j^-\rangle=
\begin{cases} \frac{2F(b_j)\sum_i f_i(b_j)g'_i(b_j)}{x_jG'(b_j)^2} & j=k\\ 
\frac{2F(b_j)\sum_i f_i(b_k)g_i(b_j)}{x_jG'(b_j)G'(b_k)(b_{j}-b_k)} & j\neq k \end{cases},$$
and so
$$-A_\xi^*(e_j^-)=\frac{\sum_i f_i(b_j)g'_i(b_j)}{G'(b_j)} e_j^- +
\sum_{k\neq j} \frac{x_kF(b_j)\sum_i f_i(b_k)g_i(b_j)}{x_jF(b_k)G'(b_j)(b_{j}-b_k)} e_k^-,$$
\begin{equation}\label{skew-sym-A-eq}
(A_\xi-A_\xi^*)(e_j^-)=
\sum_{k\neq j} \frac{x_kF(b_j)}{x_jF(b_k)G'(b_j)(b_j-b_k)}\cdot
\bigl(\sum_i f_i(b_j)g_i(b_k)+g_i(b_j)f_i(b_k)\bigr) e_k^-.
\end{equation}

Note that the restriction of $A_\xi-A_\xi^*$ to $V^+$ is determined from its restriction to $V^-$,
due to duality between $V^+$ and $V^-$. 






\section{Application to super-Schottky ideal}\label{s-8}

\subsection{Estimate on generic rank of the second variation}


First, let $C$ be a hyperelliptic curve. The results of \cite{PR} imply that certain quadratic
relations between differentials on $C$ can be deformed away from the hyperelliptic locus.
We keep the notations of the previous section. 

\begin{prop}\label{def-qu-rel-prop}
For any $(g-3)$-tuple of degree $g-3$ polynomials $(H_1(t),\ldots,H_{g-3}(t))$, the quadratic relation in $S^2H^0(C,\omega_C)$,
\begin{equation}\label{deformable-qu-rel}
\xi=\sum_{i=1}^{g-3}\bigl(H_i\frac{dt}{x}\ot H_it^2\frac{dt}{x}+H_it^2\frac{dt}{x}\ot H_i\frac{dt}{x} - 
2H_it\frac{dt}{x}\ot H_it\frac{dt}{x}\bigr)
\end{equation}
can be deformed away from the hyperelliptic locus.
\end{prop}

\begin{proof} According to \cite[Lem.\ 1.1.4]{PR}, the quadratic map
$$Q: H^0(\P^1,\OO(g-3))\to S^2H^0(C,\omega_C): f(t)\mapsto 
f\frac{dt}{x}\ot ft^2\frac{dt}{x}+ft^2\frac{dt}{x}\ot f\frac{dt}{x} - 2ft\frac{dt}{x}\ot ft\frac{dt}{x}$$
induces an isomorphism of the space $S^2H^0(\P^1,\OO(g-3))$ with the space
of quadratic relations between the differentials on $C$. Now, by \cite[Thm.\ 2.2.2]{PR},
any quadratic relation corresponding to a degenerate quadratic form on $H^0(\P^1,\OO(g-3))^*$
can be deformed away from the hyperelliptic locus. This immediately implies our assertion.
\end{proof}

\begin{theorem}\label{generic-rank-thm}
(i) Assume that $g\ge 5$ is odd. There exists a non-empty open locus $\UU\sub\MM_g$, such that
for $C\in \UU$ and a theta-characteristic $L$ over $C$ such that $h^0(L)=0$, there exists
a quadratic relation between the differentials, $\xi\in H^0(C,\om_C)^{\ot 2}$, such that
the skew-symmetrization of $m_3(?,\xi)\in {\bigwedge}^2 H^0(\om_C\ot L)$ is nondegenerate.

\noindent
(ii) In the case $g\ge 4$ is even, the statement is that for a generic $C$ there exists $\xi$ such that the 
skew-symmetrization of $m_3(?,\xi)$ has rank $\ge 2g-4$. 
\end{theorem}

\begin{proof}
(i) Let us compute the skew-symmetric form associated with the quadratic relation \eqref{deformable-qu-rel},
which we write in the form $\xi=\sum_{i=1}^{g-3}\xi_i$, where $\xi_i$ depends on $H_i$ as in the right-hand
side of \eqref{deformable-qu-rel}.
We can calculate the matrices of the operators 
$$B(i):=(A_{\xi_i}-A^*_{\xi_i})|_{V^-}$$ 
using \eqref{skew-sym-A-eq}. 
Let us consider the rescaled bases of $V^-$, 
$$d_j:=\frac{x_jG'(b_j)}{F(b_j)}e_j^-, \ \ d'_j=\frac{x_j}{F(b_j)}e_j^-, j=1,\ldots,g-1.$$
Then the coefficient of $d'_k$ in $B(\om_i,\om_j,\om_k)(d_j)$ is given by
$$B(i)_{qp}=\frac{1}{b_p-b_q}
\bigl(H_i(b_p)H_i(b_q)b_q^2+H_i(b_q)H_i(b_p)b_p^2-2H_i(b_p)H_i(b_q)b_pb_q\bigr)=H_i(b_p)H_i(b_q)(b_p-b_q).$$
We have $B(i)^t=-B(i)$ and
$$\im(B(i))=\lan \sum_k H_i(b_k) e_k, \sum_k b_kH_i(b_k) e_k\ran.$$

Now, assuming that $g$ is odd, let us consider the relation $\xi=\sum_{i=1}^{(g-1)/2} \xi_i$, where
$$H_1=1, H_2=t^2, H_3=t^4, \ldots, H_{\frac{g-1}{2}}=t^{g-3}.$$
Then the vectors spanning images of $B(i)$, $i=1,\ldots,(g-1)/2$, are evaluations of the monomials $1,t,\ldots,t^{g-2}$
at $b_1,\ldots,b_{g-1}$, so they are linearly independent.
This implies that $B=B_1+\ldots+B_{(g-1)/2}$ is nondegenerate. Hence, $A_{\xi}-A^*_{\xi}$ is nondegenerate.

Since $(g-1)/2\le g-3$, by Proposition \ref{def-qu-rel-prop}, we can find a family of curves with quadratic relation $\xi$ specializing to
the hyperelliptic curve with the above quadratic relation. Furthermore, we can extend it to a family of curves with
theta-characteristics. Thus, we will obtain a non-hyperelliptic curve $(C,L)$ with a theta-characteristic $L$ such that
$h^0(L)=0$, and a quadratic relation $\xi$ such that the skew-symmetrization of $m_3(?,\xi)$ is nondegenerate.

It remains to recall a well known fact that for non-hyperelliptic curves the product map
$$\kappa: S^2H^0(C,\om_C)\to H^0(C,\om_C^{\ot 2})$$
is surjective, so that spaces $K=\ker(\kappa)$ form a vector bundle over the non-hyperelliptic locus in $\MM_g$.
Thus, a quadratic relation $\xi$ with the above property can be found over an open locus in $\MM_g$.

\noindent
(ii) The proof is similar to that of (i), except that we consider $\xi=\sum_{i=1}^{(g-2)/2} \xi_i$, where
$$H_1=1, H_2=t^2, \ldots, H_{\frac{g-2}{2}}=t^{g-4}.$$
Then the matrix $B=B_1+\ldots+B_{(g-2)/2}$ has rank $g-2$. Hence $A_{\xi}-A^*_{\xi}$ has rank $2g-4$.
\end{proof}

\subsection{Super-Schottky ideal}

Let $\wt{\SS}^+\to \SS^+$ be the covering of $\SS^+=\SS^+_g$ corresponding to a choice of a symplectic 
basis in $H^1(C,\Z)$, and let $\wt{\UU}\sub \wt{\SS}^+$ be the preimage of the open substack
$\UU\sub\SS^+$ corresponding to theta-characteristics with trivial $H^0$.
Then we have a well defined {\it superperiod map},
$$\per:\wt{\UU}\to LG_{2g},$$
where $LG_{2g}$ is the Lagrangian Grassmannian of the $2g$-dimensional symplectic vector space over $\C$,
corresponding to the Lagrangian subbundle $\pi_*\om_{X/\wt{\UU}}$ in the trivialized symplectic bundle
$R^1\pi_*\C_{X/\wt{\UU}}$, where $X\to\wt{\UU}$ is the universal supercurve.

We define the {\it super-Schottky ideal} $\II_{s-Sch}$ on $LG_{2g}$ as
the ideal defining the schematic image of $\per$. In other words, a
local function $f$ belongs to $\II_{s-Sch}$ if and only if $\per^*f=0$.

Let $\II_{Sch}\sub LG_{2g}$ be the usual Schottky ideal corresponding
to the image of the usual period map $\wt{\MM}_g\to LG_{2g}$ (where
$\wt{\MM}_g$ is the covering of $\MM_g$ corresponding to a choice of a
symplectic basis in cohomology).  It is clear that
$\II_{s-Sch}\sub \II_{Sch}$. Furthermore, for any $f\in \II_{Sch}$,
the pull-back $\per^*f$ is even, so it belongs to the square of the
nilradical in the structure sheaf of $\wt{\SS}^+$.  Since there are
$2g-2$ odd variables on $\wt{\SS}^+$, it follows that
$$\II_{Sch}^g\sub \II_{s-Sch}.$$

\begin{theorem}\label{Schottky-thm}
  Let $d$ be the minimal number such that
  $\II_{Sch}^d\sub \II_{s-Sch}$.  Then for $g$ odd, we have $d=g$,
  while for $g$ even we have $d\ge g-1$.
\end{theorem}

\begin{proof}
  At a generic point of the classical Schottky locus, for any conormal
  vector $\xi$ we can find a function $f\in \II_{Sch}$ with $\xi$ as
  its leading part. Now Theorem \ref{generic-rank-thm} implies that
  $\per^*f^{g-1}\neq 0$ for odd $g$, while $\per^*f^{g-2}\neq 0$ for
  even $g$.
\end{proof}


\begin{thebibliography}{XX}
\bibitem{BK} A. I. Bondal, M. M. Kapranov,
  {\it Framed triangulated categories}, (Russian)
  Mat. Sb. 181 (1990), no. 5, 669--683;
  translation in Math. USSR-Sb. 70 (1991), no. 1, 93--107.
\bibitem{BHRP} U.~Bruzzo, D.~Hern\'andez Ruip\'erez, A.~Polishchuk, {\it Notes on Fundamental Algebraic Supergeometry. Hilbert and Picard superschemes}, 
arXiv:2008.00700.
\bibitem{CV} G. Codogni, F. Viviani,
  {\it Moduli and periods of supersymmetric curves},
 	Adv. Theor. Math. Phys. 23 (2019), no. 2, 345--402.
\bibitem{CraneRabin1988}
  L. Crane, J. M. Rabin,
  \emph{Super Riemann Surfaces: Uniformization and
    Teichm\"uller Theory},
  Commun. Math. Phys. 113 (1988), 601--623.
\bibitem{Deligne1987} P. Deligne,
  {\it Letter to Manin}, 25 September 1987,
  http://publications.ias.edu/deligne/paper/2638
\bibitem{DHokerPhong1989}
  E. D'Hoker, D. H. Phong,
  {\it Conformal Scalar Fields And Chiral Splitting On Super Riemann Surfaces}, Commun. Math. Phys. 125 (1989) 469--513.
\bibitem{DHokerPhong2002-1}
    E. D'Hoker, D. H. Phong,
{\it Two-loop superstrings. I. Main formulas.}
 Phys. Lett. B  529  (2002),  no. 3-4, 241--255.
		{\it II. The chiral measure on moduli space.}
 Nuclear Phys. B  636  (2002),  no. 1-2, 3--60.
		{\it
III. Slice independence and absence of ambiguities.}
 Nuclear Phys. B  636  (2002),  no. 1-2, 61--79.
	{\it	 IV. The cosmological constant and modular forms.}
 Nuclear Phys. B  639  (2002),  no. 1-2, 129--181.
\bibitem{DHokerPhong2002-2}  
   E.~D'Hoker and D.~H.~Phong, {\em Lectures on two loop superstrings},
  Conf.\ Proc.\ C 0208124, 85 (2002), [hep-th/0211111].
\bibitem{DRS} S. N. Dolgikh, A. A. Rosly, A. S. Schwarz,
  {\it Supermoduli spaces},
  Commun. Math. Phys. 135 (1990), no. 1, 91--100.
\bibitem{FP} R. Fisette and A. Polishchuk, {\it
    $A_\infty$-algebras associated with curves and rational functions
    on $\MM_{g,g}$. I,} Compos. Math. 150 (2014), no. 4, 621--667.
\bibitem{DonagiWitten2012}
 R. Donagi, E.  Witten, {\it Supermoduli space is not projected},
 String-Math 2012, 
 19--71, Proc. Sympos. Pure Math., 90, Amer. Math. Soc., Providence, RI,  2015.
\bibitem{Harris1982} J. Harris,
  {\it Theta-characteristics on
   algebraic curves},
  Trans. Amer. Math. Soc. 271 (1982), no. 2, 611--638.
\bibitem{Hart}
  R. Hartshorne, 
  {\it Algebraic Geometry},
  Springer-Verlag, New York, 1977
\bibitem{Kriv} M.~I.~Krivoruchenko, {\it Trace Identities for Skew-Symmetric Matrices}, 
Mathematics and Computer Science, Vol. 1, No. 2, 2016, 21--28.
\bibitem{LeBrunRothstein1988}
  C.~LeBrun and M.~Rohtstein,
  \emph{Moduli of Super Riemann Surfaces},
  Commun. Math. Phys. 117 (1988), 159--176.
\bibitem{Mumford1972} D. Mumford,
  {\it Theta characteristics of an algebraic curve},
  Ann. Sci. \'Ecole Norm. Sup. (4) 4 (1971), 181--192.  
\bibitem{Nagaraj1990} D. S. Nagaraj, {\it On the moduli of curves with
    theta-characteristics}, Compositio Math. 75 (1990), no. 3,
  287--297.
\bibitem{P-CYBE} 
  A.~Polishchuk,
  {\it Classical Yang-Baxter equation and the $A_\infty$-constraint},
  Adv. Math. 168 (2002), no. 1, 56--95. 
\bibitem{PR} A.~Polishchuk and E.~M.~Rains, 
{\it Hyperelliptic limits of quadrics through canonical curves and ribbons}, arXiv:1905.12113.
\bibitem{RR}  M.~J.~Rothstein, J.~M.~Rabin, {\it Abel's theorem, and Jacobi inversion for supercurves over a thick superpoint}, 
J. Geom. Phys. 90 (2015), 95--103. 
\bibitem{RoslySchwarzVoronov1989} A.~A.~Rosly, A.~S.~Schwarz,
  A.~A.~Voronov, {\it Superconformal Geometry and String Theory},
  Commun. Math. Phys. 120 (1989), 437--450.
\bibitem{Voronov}
    A.~A.~Voronov, {\it A formula for the Mumford measure in
      superstring theory}, Funct. Anal. Appl. 22 (1988), no. 2,
    139--140.
\bibitem{VMP} A.~A.~Voronov, Yu.~I.~Manin, I.~B.~Penkov, {\it Elements
    of supergeometry}, J. Soviet Math. 51 (1990), no. 1, 2069--2083.
\bibitem{Witten} E. Witten,
  {\it Notes On Super Riemann Surfaces And Their Moduli},
   Pure Appl. Math. Q. 15 (2019), no. 1, 57--211.
\bibitem{Witten-hol-str} E. Witten,
  {\it Notes on Holomorphic String And Superstring Theory Measures Of Low Genus},
  Analysis, complex geometry, and mathematical physics: in honor of Duong H. Phong, 307--359,
Contemp. Math., 644, Amer. Math. Soc., Providence, RI, 2015.
\end{thebibliography}
\end{document}